\documentclass[11pt]{amsart}  
\usepackage{amsmath,amssymb,amsthm,amsfonts}
\usepackage{mathrsfs}
\usepackage{mathtools}
\usepackage{xcolor}
\usepackage[all,cmtip]{xy}

\usepackage{geometry}
\usepackage[style=alphabetic, maxnames=99, minnames=99]{biblatex} %Imports biblatex package
\usepackage{hyperref}
\hypersetup{
    colorlinks,
    linkcolor={black},
    citecolor={black},
    urlcolor={black}
}
\usepackage{xfrac}

\theoremstyle{plain} %% This is the default, anyway
\newtheorem{thm}{Theorem}[section]

\newtheorem{lem}[thm]{Lemma}
\newtheorem{prop}[thm]{Proposition}
\newtheorem{defn}[thm]{Definition}
\theoremstyle{remark}
\newtheorem{rem}[thm]{Remark}

\numberwithin{equation}{section}

\newcommand{\R}{\mathbb{R}}
\newcommand{\C}{\mathbb{C}}

\newcommand{\E}{\mathop{\mathbb{E}}}
\newcommand{\Z}{\mathbb{Z}}

\DeclarePairedDelimiter\abs{\lvert}{\rvert}
\DeclarePairedDelimiter\norm{\lVert}{\rVert}

\makeatletter
\newcommand{\vast}{\bBigg@{4}}
\newcommand{\Vast}{\bBigg@{5}}
\makeatother

\newcommand{\sqw}[1]{\sum_{#1=\square\,\mathrm{mod}\,W}}
\newcommand{\sqp}[1]{\sum_{#1=\square\,\mathrm{mod}\,p^n}}
\newcommand{\sqwone}[1]{\sum_{#1=\square\,\mathrm{mod}\,W_1}}
\newcommand{\sqwtwo}[1]{\sum_{#1=\square\,\mathrm{mod}\,W_2}}
\newcommand{\aw}[1]{\sum_{#1\,\mathrm{mod}\,W}}
\newcommand{\ap}[1]{\sum_{#1\,\mathrm{mod}\,p^n}}
\newcommand{\awone}[1]{\sum_{#1\,\mathrm{mod}\,W_1}}
\newcommand{\awtwo}[1]{\sum_{#1\,\mathrm{mod}\,W_2}}

\title[The density version of Waring's problem and the Waring--Goldbach problem]{On the density version of quadratic Waring's problem and the quadratic Waring--Goldbach problem}

\author{Zi Li Lim}
\address{Department of Mathematics, UCLA, Los Angeles, CA 90095, USA}
\email{zililim@math.ucla.edu}

\begin{document}

\maketitle

\begin{abstract}
We prove a sharp density theorem for quadratic Waring's problem over cyclic groups, when the number of variables is at least $5$. Also, we obtain some new improvements on the density version of the quadratic Waring--Goldbach problem over integers.
\end{abstract}

\section{Introduction}\label{sec: intro}

\subsection{Introduction}

Denote the set of all prime numbers by $\mathcal{P}$. Let $A$ be a subset of primes, recall that the relative lower density $\delta_{\mathcal{P}}(A)$ is defined to be
\begin{equation*}
    \delta_{\mathcal{P}}(A)=\liminf_{N\rightarrow\infty}\frac{|A\cap[N]|}{|\mathcal{P}\cap[N]|}.
\end{equation*}
Li and Pan \cite{LiPan} proved a density version of Vinogradov three primes theorem, which states that if $\delta_{\mathcal{P}}(A)>2/3$, then every sufficiently large odd integer can be represented as a sum of three elements in $A$. In this paper, we refer to quantities such as the $2/3$ in Li and Pan's result as a ``density threshold''.

In a breakthrough paper, Shao \cite{Shao} proved the sharp density version of Vinogradov three primes theorem, lowering the density threshold from $2/3$ to $5/8$, which led to many interesting subsequent developments, such as \cite{ShaoLFunction}, \cite{MatomakiShao} and \cite{MatomakiMaynardShao}. Later, density type theorems were considered in other arithmetic settings, for example, Waring's problem \cite{Salmensuu}, or the Waring--Goldbach problem \cite{Gao}, \cite{zhao2025densitytheoremprimesquares}, \cite{Tan}, but the optimal densities are unknown.

In this paper, we present some refinements of density type theorems for quadratic Waring's problem and the quadratic Waring--Goldbach problem. Our first theorem is a sharp density theorem for quadratic Waring's problem over cyclic groups. Informally, it gives a density threshold of $1/s$ of the Waring problem for sums of $s$ squares in a cyclic group, provided that the order of the group is not divisible by small primes, and $s$ is at least $5$. For any positive integer $W$, denote the squares in $\Z/W$ by $(\Z/W)^{(2)}$. 

\begin{thm}\label{main}
    Let $s\geq 5$ be an integer and $\theta\in (1/s,1]$. There exists a constant $M(s,\theta)$ depending only on $s$ and $\theta$ such that the following holds.\\

    Suppose $W=\prod_{i} p_i^{n_i}$ is an integer with $p_i\geq M(s,\theta)$ for all $i$, and $A$ is a subset of the squares in $\mathbb{Z}/W$ with relative density greater than or equal to $\theta$, i.e. $|A|\geq\theta |(\mathbb{Z}/W)^{(2)}|$. Then, for all $y\in \mathbb{Z}/W$, there exist $x_1,x_2,...,x_s\in A$ such that $y=x_1+x_2+...+x_s$. Moreover, one has the following quantitative estimate
    \begin{equation*}
        |\{ (x_1,x_2,...,x_s)\in A^s:y=x_1+x_2+...+x_s\}|\geq c(s,\theta)W^{-1} |(\mathbb{Z}/W)^{(2)}|^s
    \end{equation*}
    for some constant $c(s,\theta)>0$ depending only on $s$ and $\theta$.
\end{thm}

\begin{rem}
    For Theorem \ref{main}, it is necessary to exclude small prime divisors. Let $W$ be an integer coprime to $3$, and $A=\{x\in (\Z/3W)^{(2)}:x\equiv 0\;\mathrm{mod}\;3\}$, then the relative density of $A$ is $1/2$, but the $s$-fold sumset of $A$ can only contains the elements congruent to $0\;\mathrm{mod}\;3$.
\end{rem}

It might be worth to point out that the squares in $\Z/W$ are fairly sparse. Suppose $W=\prod_{1\leq i\leq r}p_i$ is a product of $r$ distinct odd primes, then the absolute density of $(\Z/W)^{(2)}$ in the ambient space $\Z/W$ is at most $\prod_{1\leq i\leq r}(p_i+1)/(2p_i)\leq (2/3)^r$, which tends to zero when $r\rightarrow\infty$. Hence, a subset of $(\Z/W)^{(2)}$ with relative density greater than $1/s$ can have absolute density arbitrarily close to zero.

The next theorem confirms that the density threshold $1/s$ is indeed sharp, except possibly at the endpoint.

\begin{thm}\label{counter examples}
    Let $s\geq 5$ be an integer and $\tau\in [0,1/s)$. For every sufficiently large (depending only on $s$ and $\tau$) prime $p$, there exists a subset $A_p$ of the squares in $\mathbb{Z}/p$ with relative density greater than or equal to $\tau$, i.e. $|A_p|\geq \tau|(\mathbb{Z}/p)^{(2)}|$, but the $s$-fold sumset $sA_p$ is not equal to $\mathbb{Z}/p$.  
\end{thm}

Now, we turn our attention from cyclic groups to integers. Salmensuu \cite{Salmensuu} proved a density theorem for Waring's problem. In quadratic case of Salmensuu's theorem, the number of squares needed is at least $44$ and the density threshold is approximately $0.9206$. Gao \cite{Gao}, Zhao \cite{zhao2025densitytheoremprimesquares} and Tan \cite{Tan} generalized Salmensuu's theorem to the Waring--Goldbach problem.

We choose to work with the Waring--Goldbach problem to demonstrate our method for technical reasons (it is easier to work with reduced congruence classes for corresponding local problem, see Theorem \ref{main for almost all moduli}). We obtain some new improvements, but not the optimal density. For any integer $s\geq 5$, define 
\begin{equation*}
    D_s\coloneqq\begin{dcases}
        \frac{59}{60}&\text{if $s=5$}\\
        \frac{7}{8}&\text{if $s=6$}\\
        \frac{3}{4}&\text{if $s=7$}\\
        \frac{s+13}{4s}&\text{if $8\leq s\leq 12$}\\
        \frac{1}{2}&\text{if $s\geq 13$}
    \end{dcases}
\end{equation*}

\begin{thm}\label{main for primes}
    Let $s\geq 5$ be an integer, and $A$ be a subset of primes with relative lower density $\delta_{\mathcal{P}}(A)>\sqrt{D_s}$. Then, for all sufficiently large integers $y$ that are congruent to $s\;\mathrm{mod}\;24$, there exist primes $p_1,p_2,...,p_s\in A$ such that $y=p_1^2+p_2^2+...+p_s^2$.
\end{thm}

\begin{rem}
    The congruence condition in Theorem \ref{main for primes} is necessary, since the square of any prime greater than $3$ is congruent to $1\;\mathrm{mod}\;24$.
\end{rem}

For comparison, Gao \cite{Gao} proved a density theorem for the quadratic Waring--Goldbach problem for $s\geq 44$ and density threshold $3/4$, and Zhao \cite{zhao2025densitytheoremprimesquares} improved the result to $s\geq 8$ and density threshold $\sqrt{1-\min\{s,16\}/32}$. (Gao also considered the Waring--Goldbach problem for higher powers, see relevant discussions later.) The density thresholds $\sqrt{D_s}$ given by Theorem \ref{main for primes} improve upon previous records when $s\geq 6$. However, when $s=5$, our density threshold is not as good as the density threshold $\sqrt{4/5}$ obtained by Tan \cite{Tan}. (We still include the computations for Theorem \ref{main for primes} when $s=5$ to demonstrate the method.) The table below is a detailed numerical comparison between the density thresholds in Theorem \ref{main for primes} and the previous records.

\renewcommand{\arraystretch}{2}
\begin{center}
\begin{tabular}{|c|c|c||c|c|c|} 
 \hline
 $s$ & $\sqrt{D_s}$ & Previous record& $s$ & $\sqrt{D_s}$ & Previous record\\  
 \hline
 $5$ & $\sqrt{\frac{59}{60}}\approx 0.9916$ & $\sqrt{\frac{4}{5}}\approx 0.8944$ &$11$ & $\sqrt{\frac{6}{11}}\approx 0.7385$ & $\sqrt{\frac{21}{32}}\approx 0.8101$\\
 \hline
 $6$ & $\sqrt{\frac{7}{8}}\approx 0.9354$ & unknown previously &$12$ & $\sqrt{\frac{25}{48}}\approx 0.7217$ & $\sqrt{\frac{5}{8}}\approx 0.7906$\\
 \hline
 $7$ & $\sqrt{\frac{3}{4}}\approx 0.8660$ & unknown previously &$13$ & $\sqrt{\frac{1}{2}}\approx 0.7071$ & $\sqrt{\frac{19}{32}}\approx 0.7706$\\
 \hline
 $8$ & $\sqrt{\frac{21}{32}}\approx 0.8101$ & $\sqrt{\frac{3}{4}}\approx 0.8660$ &$14$ & $\sqrt{\frac{1}{2}}\approx 0.7071$ & $\sqrt{\frac{9}{16}}\approx 0.7500$\\
 \hline
 $9$ & $\sqrt{\frac{11}{18}}\approx 0.7817$ & $\sqrt{\frac{23}{32}}\approx 0.8478$ &$15$ & $\sqrt{\frac{1}{2}}\approx 0.7071$ & $\sqrt{\frac{17}{32}}\approx 0.7289$\\
 \hline
 $10$ & $\sqrt{\frac{23}{40}}\approx 0.7583$ & $\sqrt{\frac{11}{16}}\approx 0.8292$ &$\geq 16$ & $\sqrt{\frac{1}{2}}\approx 0.7071$ & $\sqrt{\frac{1}{2}}\approx 0.7071$\\
 \hline
\end{tabular}
\end{center}

(In the table above, the previous records are obtained by Zhao when $s\geq 8$ and by Tan when $s=5$.)

The key to Theorem \ref{main for primes} is the following local result. For any integer $s\geq 5$, define
\begin{equation*}
    d_s\coloneqq\begin{dcases}
        \frac{9}{10}&\text{if $s=5$}\\
        \frac{5}{6}&\text{if $s=6$}\\
        \frac{s+13}{4s}&\text{if $7\leq s\leq 12$}\\
        \frac{1}{2}&\text{if $s\geq 13$}
    \end{dcases}
\end{equation*}
Note that $D_s\geq d_s\geq 1/2$ and $d_s\geq (s+13)/(4s)$ for all $s$. Similar as before, for any positive integer $W$, denote the squares in $(\Z/W)^{\times}$ by $(\Z/W)^{\times(2)}$. 

\begin{thm}\label{main for almost all moduli}
    Let $s\geq 5$ be an integer, and $W$ be a square free integer. Assume $(2\cdot3\cdot5,W)=1$ when $s\geq 6$, and $(2\cdot3\cdot5\cdot7,W)=1$ when $s=5$. Let $f_j:(\Z/W)^{\times(2)}\longrightarrow[0,1]$ be functions for $1\leq j\leq s$ with
    \begin{equation*}
        \text{$\frac{1}{s}(\E[f_1]+\E[f_2]+...+\E[f_s])>d_s$ and $\E[f_j]\neq 0$ for all $1\leq j\leq s$.}
    \end{equation*}
    Then, for all $y\in \Z/W$, we have $y=x_1+x_2+...+x_s$ for some $x_1,x_2,...,x_s\in (\Z/W)^{\times(2)}$ such that
    \begin{equation*}
        \text{$\frac{1}{s}(f_1(x_1)+f_2(x_2)+...+f_s(x_s))>d_s$ and $f_j(x_j)\neq 0$ for all $1\leq j \leq s$.}
    \end{equation*}
\end{thm}

The essential distinction between Theorem \ref{main} and Theorem \ref{main for almost all moduli} is Theorem \ref{main} excludes considerably more amount of small primes. Due to this distinction, their proofs are fundamentally different.

\subsection{Sketch of the proofs}

All density type theorems, including ours, crucially rely on the transference principle pioneered by Green and Tao in their celebrated proof that primes contain arbitrarily long arithmetic progressions \cite{Green}, \cite{GreenTao}. 

One of the key ingredients needed for the transference principle is restriction estimates. For Theorem \ref{main}, since we are working over cyclic groups $\Z/W$, we need restriction estimates for squares in $\Z/W$ instead of the classical Bourgain's restriction estimates for squares in $\Z$ \cite{Bourgain}. Surprisingly, the proofs are different in the cyclic group setting and in the integer setting---the squares are denser in cyclic groups, but one has extra product structure of the groups.

For Theorem \ref{main for primes}, our proof is based on a variant of the combinatorial inequalities developed by Shao \cite{Shao} and by Lacey, Mousavi, Rahimi and Vempati \cite{lacey2024densitytheoremhigherorder}. The combinatorial inequality is inductive in nature, and the base cases are handled by linear programming with computer assistance.

\subsection{Future directions}

The main purpose of this paper is demonstrating the methods as a proof of concept, rather than pursuing the best density that can be obtained by the methods. There are certainly rooms for further improvements. First, we expect the density threshold $\sqrt{D_s}$ can be improved to $D_s$ by not using certain smooth weights for majorants (see Appendix \ref{appendix 1}). Second, the numerical linear programming computations in Section \ref{sec: combinatorial lemma} can be scaled up to lead to better density.

In another direction, we believe Theorem \ref{main} and Theorem \ref{main for primes} can be generalized to higher powers. We also expect the techniques for Theorem \ref{main for primes} are applicable for the setting of Waring's problem.

\begin{rem}
    To prove results like Theorem \ref{main for almost all moduli}, in principle, one can use Theorem \ref{main} to handle large primes, formulate the problem for medium primes in linear programming, and deal with small primes by hand. However, the number of ``medium" primes here is astronomical, even for computers.  
\end{rem}

\begin{rem}
    By classical estimates in number theory and union bounds, it is possible to prove Theorem \ref{main for primes} for density threshold extremely close to $1$.
\end{rem}

\subsection{Road map}

This paper is organized as follows. We review some basic facts about the squares in $\Z/W$ in Section \ref{sec: preliminaries}. Section \ref{sec: restriction} contains the restriction estimates for squares in $\Z/W$. Section \ref{sec: sumsets} is about some auxiliary additive combinatorics results. We use the transference principle to prove Theorem \ref{main} in Section \ref{sec: transference principle}. The examples for Theorem \ref{counter examples} are constructed in Section \ref{sec: counter examples}.

The key combinatorial inequality for Theorem \ref{main for primes} is established in Section \ref{sec: combinatorial lemma}. Section \ref{sec: small moduli} extends Theorem \ref{main for almost all moduli} to all moduli. For self-containedness, we include two appendices. Appendix \ref{appendix} contains some Gauss sum type results and Appendix \ref{appendix 1} is about how to reduce Theorem \ref{main for primes} to Theorem \ref{main for all moduli}.

\subsection{Acknowledgments}

The author would like to thank his advisor Terence Tao for his guidance and support, especially for his insights on restriction estimates for squares in cyclic groups. The author would also like to thank Xuancheng Shao for encouragement and pointing out references.

\section{Preliminaries}\label{sec: preliminaries}

In this section, we will clarify our notation and review some basic facts about the squares in $\Z/W$. A comprehensive discussion on this subject can be found in the excellent book \cite[Chapter 3]{ANT}.

\subsection{Squares in $\Z/W$}
Let $W$ be a positive integer. Denote the squares in $\Z/W$ by $(\Z/W)^{(2)}$, more formally
\begin{equation*}
    (\Z/W)^{(2)}=\{x\in\Z/W:x=y^2\,\text{for some $y\in\Z/W$}\}
\end{equation*}
Similarly, we define
\begin{equation*}
    (\Z/W)^{\times(2)}=\{x\in\Z/W:x=y^2\,\text{for some $y\in(\Z/W)^{\times}$}\}
\end{equation*}
Let $S(W)$ be the number of squares in $\Z/W$, that is
\begin{equation*}
    S(W)=|(\Z/W)^{(2)}|.
\end{equation*}
We will use the convenient notation $\sqw{x}$ to denote the summation over the squares in $\Z/W$, so $\sqw{x}$ and $\sum_{x\in(\Z/W)^{(2)}}$ are interchangeable.

Now we recall some classical and well-known results regarding the squares. Here and throughout, $p$ will always be a prime number. We start with the number of sqaures in $\Z/p^n$.  

\begin{prop}\label{number of squares}
    Let $p\geq 3$ and $n\geq 1$, then
    \begin{equation*}
        S(p^n)=\begin{dcases}
            \frac{1}{2}\cdot\frac{p^{n+1}-1}{p+1}+1&\text{when $n$ is odd}\\            
            \frac{1}{2}\cdot\frac{p(p^n-1)}{p+1}+1&\text{when $n$ is even}
        \end{dcases}
    \end{equation*}
    In particular, we have $S(p^n)\asymp p^n$.
\end{prop}

Interested readers can find a proof of Proposition \ref{number of squares} in \cite{Stangl}. In fact, instead of a closed form formula, we will only need $S(p^n)=\frac{1}{2}p^n+O(p^{n-1})$, which can be derived more straightforwardly by looking at the squares in $(\Z/p^n)^{\times}$. Intuitively, Proposition \ref{number of squares} says that roughly half of the elements of $\Z/p^n$ are squares. By Chinese remainder theorem, $(\Z/W_1W_2)^{(2)}$ can be identified with $(\Z/W_1)^{(2)}\times(\Z/W_2)^{(2)}$ if $W_1$ and $W_2$ are coprime, hence Proposition \ref{number of squares} also provides a formula for $S(W)$ for general modulus $W$. The next proposition is a Gauss sum type result.

\begin{prop}\label{Gauss}
    Let $p\geq3,n\geq1$ and $t\in \Z$, then 
    \begin{equation*}
        \abs[\bigg]{\sqp{x}e_{p^n}(tx)}\ll (t,p^n)p^{1/2}
    \end{equation*}
    where the implied constant does not depend on $p,n$ or $t$.
\end{prop}

The formulation of Proposition \ref{Gauss} is slightly different than the Gauss sum type results in literature. To save readers' effort, we briefly explain how to derive Proposition \ref{Gauss} from standard references in Appendix \ref{appendix}. Intuitively, the term $(t,p^n)$ corresponds to the repetition of some elements in the sum, and the term $p^{1/2}$ corresponds to the cancellation.

\subsection{Other notation}
We will use asymptoptic notation $\ll,\gg,\asymp,O,o$, and all the implied constants are allowed to depend on $s,\theta$.

Let $W$ be a positive integer. We will denote $e^{\frac{2\pi i}{W}\cdot}$ by $e_W(\cdot)$, and $e^{2\pi i\cdot}$ by $e(\cdot)$. For functions on $\Z/W$, all $l^q$-norms, Fourier transforms and convolutions are taken with respect to the counting measure. Hence, if $f,g:\Z/W\longrightarrow\C$ are functions, then the $l^q$-norm is
\begin{equation*}
    \norm{f}_q=\left(\sum_{x\in\Z/W}\abs{f(x)}^q\right)^{1/q},
\end{equation*}
the Fourier transform is
\begin{equation*}
    \widetilde{f}(\xi)=\sum_{x\in\Z/W}f(x)e_W(-x\xi),
\end{equation*}
and the convolution is
\begin{equation*}
    f*g(x)=\sum_{\substack{y,z\in\Z/W\\y+z=x}}f(y)g(z).
\end{equation*}

For functions on $\Z$, the $l^q$-norms are taken with respect to counting measure. For functions on $\R/\Z$, the $L^q$-norms are taken with respect to Lebesgue measure. If $f,g$ are finitely supported functions on $\Z$, then the Fourier transform is defined to be
\begin{equation*}
    \widehat{f}(\alpha)=\sum_{n\in \Z}f(n)e(n\alpha)
\end{equation*}
and the convolution is defined to be
\begin{equation*}
    f*g(x)=\sum_{\substack{y,z\in \Z\\y+z=x}}f(y)g(z)
\end{equation*}

Finally, we will use $\norm{\cdot}_{\R/\Z}$ to denote the distance to the nearest integer.

\section{Restriction Estimates}\label{sec: restriction}

In this section, we formulate and prove restriction estimates for squares in $\Z/W$, which is an ingredient needed for the transference principle. It can be viewed as an analog of Bourgain's restriction estimates for squares in integers \cite{Bourgain} over $\Z/W$. However, the proof of the cases over $\Z/W$ and that of the cases over integers are somewhat different, since one can explore the product structure of $\Z/W$ via Chinese remainder theorem.

Recall that we denote the number of squares in $\Z/W$ by $S(W)$, and the notation $\sqw{x}$ means the summation over the squares in $\Z/W$. See Section \ref{sec: preliminaries} for more on our notation.

\begin{defn}
    Let $W$ be a positive integer and $q\geq 2$ be a real number. Define $C_q(W)$ to be the best constant such that
    \begin{equation*}
        \left(\aw{t}\norm[\bigg]{\sqw{x} a_x e_W(tx)}^q_H\right)^{1/q}\leq C_q(W) S(W)^{1/2} \left(\sqw{x}\norm{a_x}^2_H\right)^{1/2}
    \end{equation*}
    for all finite dimensional Hilbert space $H$ with norm $\lVert\cdot\rVert_H$ and all vectors $a_x\in H$.
\end{defn}

\begin{rem}
    The constants $C_q(W)$ are defined for vector-valued restriction estimates. When specializing to scalars, they give more ``natural'' inequalities:
    \begin{equation*}
        \left(\aw{t}\abs[\bigg]{\sqw{x} a_x e_W(tx)}^q\right)^{1/q}\leq C_q(W) S(W)^{1/2} \left(\sqw{x}\abs{a_x}^2\right)^{1/2}
    \end{equation*}
    for all complex numbers $a_x$. Indeed, only this form of restriction estimates would be needed for the proof of our main theorem. However, the vector-valued formulation is crucial for the multiplicativity of $C_q(W)$ (Lemma \ref{multiplicativity}).
\end{rem}

The purpose of this section is to prove the following theorem, which says that one has satisfactory restriction estimates for squares in $\Z/W$ for $q>4$.

\begin{thm}\label{restriction for W}
    Let $q>4$ be a real number. There exists a constant $M(q)$ depending only on $q$ such that the following holds.
    
    Suppose $W=\prod_i p_i^{n_i}$ is a positive integer with $p_i\geq M(q)$ for all $i$, then 
    \begin{equation*}
        C_q(W)=O_q(1).
    \end{equation*}
\end{thm}

As a remark, the requirement that $s\geq 5$ in Theorem \ref{main} and Theorem \ref{main for primes} is due to the critical exponent of restriction estimates for squares in $\Z/W$. The proof of Theorem \ref{restriction for W} would be deferred to the end of the section. First, we need the following lemma to reduce to the cases of prime power moduli.

\begin{lem}\label{multiplicativity}
    Let $W_1,W_2$ be two coprime positive integers, then $C_q(W_1W_2)=C_q(W_1)C_q(W_2)$.
\end{lem}

\begin{proof}
    First, we will show that $C_q(W_1W_2)\leq C_q(W_1)C_q(W_2)$. Since $W_1$ and $W_2$ are coprime, we have $\beta_2 W_1+\beta_1 W_2=1$ for some integers $\beta_1,\beta_2$ with $(\beta_1,W_1)=(\beta_2,W_2)=1$. Hence, we have multiplicative property $e_W(u)=e_W((\beta_2 W_1+\beta_1 W_2)u)=e_{W_1}(\beta_1 u)e_{W_2}(\beta_2 u)$ for all $u$. 
    
    Furthermore, by Chinese remainder theorem, we can identify $\Z/W$ with $\Z/W_1\times\Z/W_2$ and $(\Z/W)^{(2)}$ with $(\Z/W_1)^{(2)}\times(\Z/W_2)^{(2)}$. Under these identifications, we will write $x=(x_1,x_2)$ for $x\in \Z/W$ as our notation. Thus, we have

    \begin{align*}
        &\left(\aw{t}\norm[\bigg]{\sqw{x}a_x e_W(tx)}^q_H\right)^{1/q}\\
        &=\left(\awone{t_1}\awtwo{t_2}\norm[\bigg]{\sqwtwo{x_2}\sqwone{x_1}a_{(x_1,x_2)} e_{W_1}(\beta_1 t_1 x_1)e_{W_2}(\beta_2 t_2 x_2)}^q_H\right)^{1/q}.
    \end{align*}

    By definition of $C_q(W_2)$, the expression above is bounded by

    \begin{equation}\label{eqs: w1w2}
        C_q(W_2)S(W_2)^{1/2}\left(\awone{t_1}\left(\sqwtwo{x_2}\norm[\bigg]{\sqwone{x_1}a_{(x_1,x_2)}e_{W_1}(t_1x_1)}^2_H\right)^{q/2}\right)^{1/q}.
    \end{equation}

    To proceed, we will work over the vector space $H^{(\Z/W_2)^{(2)}}$. Let $b_{x_1}=(a_{(x_1,x_2)})_{x_2\in (\Z/W_2)^{(2)}}\in H^{(\Z/W_2)^{(2)}}$. Note that

    \begin{equation*}
        \norm[\bigg]{\sqwone{x_1} b_{x_1}e_{W_1}(t_1x_1)}_{H^{(\Z/W_2)^{(2)}}}=\left(\sqwtwo{x_2}\norm[\bigg]{\sqwone{x_1}a_{(x_1,x_2)}e_{W_1}(t_1x_1)}^2_H\right)^{1/2}.
    \end{equation*}

    Hence, we deduce that \eqref{eqs: w1w2} is bounded by

    \begin{equation*}
        \begin{split}            &C_q(W_2)S(W_2)^{1/2}C_q(W_1)S(W_1)^{1/2}\left(\sqwone{x_1}\norm{b_{x_1}}^2_{H^{(\Z/W_2)^{(2)}}}\right)^{1/2}\\
        &=C_q(W_1)C_q(W_2)S(W)^{1/2}\left(\sqwone{x_1}\sqwtwo{x_2}\norm{a_{(x_1,x_2)}}^2_H\right)^{1/2},
        \end{split}
    \end{equation*}
    and this concludes that $C_q(W)\leq C_q(W_1)C_q(W_2)$.

    Next, we will show that $C_q(W_1)C_q(W_2)\leq C_q(W)$. Suppose the worst examples for $C_q(W_1)$ and $C_q(W_2)$ are given by some vectors $a_{x_1}^{(1)}$ and $a_{x_2}^{(2)}$ in some vector spaces $H_1$ and $H_2$ respectively. Considering $a_{x_1}^{(1)}\otimes a_{x_2}^{(2)}$ in the tensor product $H_1\otimes H_2$, we have

    \begin{equation*}
    \begin{split}
        &\awone{t_1}\awtwo{t_2}\norm[\bigg]{\sqwone{x_1}\sqwtwo{x_2}a_{x_1}^{(1)}\otimes a_{x_2}^{(2)}e_{W_1}(t_1x_1)e_{W_2}(t_2x_2)}^q_{H_1\otimes H_2}\\
        &=\left(\awone{t_1}\norm[\bigg]{\sqwone{x_1}a_{x_1}^{(1)}e_{W_1}(t_1x_1)}^q_{H_1}\right)\left(\awtwo{t_2}\norm[\bigg]{\sqwtwo{x_2}a_{x_2}^{(2)}e_{W_2}(t_2x_2)}^q_{H_2}\right)
    \end{split}        
    \end{equation*}
    since $\norm{v_1\otimes v_2}_{H_1\otimes H_2}=\norm{v_1}_{H_1}\norm{v_2}_{H_2}$ for all $v_1\in H_1,v_2\in H_2$. This implies that $C_q(W_1)C_q(W_2)\leq C_q(W)$.
    
\end{proof}

The next proposition is the key to the proof of Theorem \ref{restriction for W}. It is about restriction estimates for squares in $\Z/p^n$.

\begin{prop}\label{restriction for pn}
    Let $q\geq 4$ be a real number, we have
    \begin{equation*}
        C_q(p^n)\leq 1+O_q \left( \frac{1}{p^{\frac{q}{2}-1}} \right)
    \end{equation*}
    for all sufficiently large (depending on $q$) primes $p$ and all positive integers $n$. The implied constant for big $O$ only depends on $q$ but not $p$ and $n$.
\end{prop}

\begin{proof}[Proof of Proposition \ref{restriction for pn} for $q=4$]
        Throughout the proof, we will assume $p$ is sufficiently large. Expanding the powers, we have
    \begin{equation*}
    \begin{split}
        \ap{t}\norm[\bigg]{\sqp{x}a_xe_{p^n}(tx)}^4_H&=\ap{t}\sqp{x,y,z,w}\langle a_x,a_y \rangle_H \langle a_z,a_w \rangle_H e_{p^n}(t(x-y+z-w))\\
        &=p^n \sum_{\substack{x,y,z,w=\square\,\mathrm{mod}\,p^n\\x-y=w-z}} \langle a_x,a_y \rangle_H \langle a_z,a_w \rangle_H .
    \end{split}       
    \end{equation*}

    Rewrite the expression above as
    \begin{equation*}
        p^n \ap{l}\sum_{\substack{x,y,z,w=\square\,\mathrm{mod}\,p^n\\x-y=l=w-z}} \langle a_x,a_y \rangle_H \langle a_z,a_w \rangle_H=p^n \ap{l}\abs[\bigg]{\sum_{\substack{x,y=\square\,\mathrm{mod}\,p^n\\x-y=l}} \langle a_x,a_y \rangle_H}^2 ,
    \end{equation*}
    and we will divide into two cases according to the invertibility of $l$.

    First, we consider the summation over $l$ that are not invertible in $\Z/p^n$. Note that
    \begin{equation*}
        p^n \sum_{\substack{l\,\mathrm{mod}\,p^n\\l\notin (\Z/p^n)^{\times}}}\abs[\bigg]{\sum_{\substack{x,y=\square\,\mathrm{mod}\,p^n\\x-y=l}} \langle a_x,a_y \rangle_H}^2 \leq p^n \sum_{\substack{l\,\mathrm{mod}\,p^n\\l\notin (\Z/p^n)^{\times}}}\left(\sqp{x} \abs{\langle a_x,a_{x-l} \rangle_H}\right)^2
    \end{equation*}
    if we conveniently denote $a_u=0$ for $u$ that is not a square mod $p^n$.

    Applying Cauchy--Schwarz inequality twice, the expression above is bounded by
    \begin{equation}\label{eqs: q is four 1}
        \begin{split}
            p^n \sum_{\substack{l\,\mathrm{mod}\,p^n\\l\notin (\Z/p^n)^{\times}}}\left(\sqp{x} \norm{a_x}_H \norm{a_{x-l}}_H\right)^2
            &\leq p^n \sum_{\substack{l\,\mathrm{mod}\,p^n\\l\notin (\Z/p^n)^{\times}}}\left(\sqp{x} \norm{a_x}_H^2\right)^2\\
            &=p^{2n-1}\left(\sqp{x} \norm{a_x}_H^2\right)^2\\
            &\ll \frac{1}{p}S(p^n)^2\left(\sqp{x} \norm{a_x}_H^2\right)^2
        \end{split}
    \end{equation}
    where the last line follows from Proposition \ref{number of squares}.

    Next, we consider the summation over $l$ that are invertible in $\Z/p^n$. To proceed, we need a little digression on number of the solutions to certain equation over $\Z/p^n$. For any $l\in (\Z/p^n)^{\times}$, define
    \begin{equation*}
        T(l,p^n)=\{(x,y)\in \Z/p^n\times \Z/p^n:x\;\text{and}\;y\;\text{are squares, and}\;x-y=l\}
    \end{equation*}
    We claim that $|T(l,p^n)|\leq \frac{p^n-p^{n-1}}{4}+2$ for $p\geq 3$. To see this, let
    \begin{equation*}
        \mathcal{A}(l,p^n)=\{(a,b)\in \Z/p^n\times \Z/p^n:a,b\neq 0\;\text{and}\;a^2-b^2=l\}
    \end{equation*}

    Note that for any $(a,b)\in \mathcal{A}(l,p^n)$, we have $(a+b)(a-b)=l$. Since $p\neq 2$, we can perform change of variables $u=a+b,v=a-b$ and rewrite the equation as $uv=l$. The assumption that $l$ is invertible implies that $u,v$ are invertible. Thus, there are at most $|(\Z/p^n)^{\times}|=p^n-p^{n-1}$ choices of $(u,v)$ and hence $|\mathcal{A}(l,p^n)|\leq p^n-p^{n-1}$.

    The relation between $\mathcal{A}(l,p^n)$ and $T(l,p^n)$ is not exactly one-to-one. For any $(x,y)\in T(l,p^n)$ with $x$ and $y$ both nonzero, there are at least four different elements in $\mathcal{A}(l,p^n)$ corresponding to it, since $(\pm a,\pm b)$ will correspond to the same $(x,y)$. Also, we need to consider the special case that $(l,0)$ and $(0,-l)$ might belong to $T(l,p^n)$. Hence, we have 
    \begin{equation*}
        |T(l,p^n)|\leq \frac{|\mathcal{A}(l,p^n)|}{4}+2\leq \frac{p^n-p^{n-1}}{4}+2.
    \end{equation*}

    Now, we can continue our previous estimates. Note that
    \begin{equation*}
        p^n \sum_{\substack{l\,\mathrm{mod}\,p^n\\l\in (\Z/p^n)^{\times}}}\abs[\bigg]{\sum_{\substack{x,y=\square\,\mathrm{mod}\,p^n\\x-y=l}} \langle a_x,a_y \rangle_H}^2 \leq p^n \sum_{\substack{l\,\mathrm{mod}\,p^n\\l\in (\Z/p^n)^{\times}}} |T(l,p^n)|\sum_{\substack{x,y=\square\,\mathrm{mod}\,p^n\\x-y=l}} \abs{\langle a_x,a_y \rangle_H}^2
    \end{equation*}
    by Cauchy--Schwarz inequality.

    Invoking the estimates for $|T(l,p^n)|$, the expression above is bounded by
    \begin{equation}\label{eqs: q is four 2}
        \begin{split}
            &p^n \left(\frac{p^n-p^{n-1}}{4}+2\right)\sum_{\substack{l\,\mathrm{mod}\,p^n\\l\in (\Z/p^n)^{\times}}} \sum_{\substack{x,y=\square\,\mathrm{mod}\,p^n\\x-y=l}} \abs{\langle a_x,a_y \rangle_H}^2\\
            &\leq p^n \left(\frac{p^n-p^{n-1}}{4}+2\right)\sum_{\substack{l\,\mathrm{mod}\,p^n\\l\in (\Z/p^n)^{\times}}} \sum_{\substack{x,y=\square\,\mathrm{mod}\,p^n\\x-y=l}} \left(\norm{a_x}_H\norm{a_y}_H\right)^2\\
            &\leq p^n \left(\frac{p^n-p^{n-1}}{4}+2\right)\sqp{x,y} \left(\norm{a_x}_H\norm{a_y}_H\right)^2\\
            &= p^n \left(\frac{p^n-p^{n-1}}{4}+2\right)\left(\sqp{x} \norm{a_x}_H^2\right)^2.
        \end{split}
    \end{equation}

    Finally, by Proposition \ref{number of squares}, we have
    \begin{equation}\label{eqs: q is four 3}
        p^n \left(\frac{p^n-p^{n-1}}{4}+2\right)S(p^n)^{-2}\leq \frac{p^{2n+2}+O(p^{2n+1})}{p^{2n+2}-O(p^{2n+1})}\leq 1+O\left(\frac{1}{p}\right).
    \end{equation}

    Combining \eqref{eqs: q is four 1}, \eqref{eqs: q is four 2} and \eqref{eqs: q is four 3}, we conclude the desired inequality $C_4(p^n)\leq 1+O(1/p)$ for sufficiently large $p$.
\end{proof}

\begin{proof}[Proof of Proposition \ref{restriction for pn} for $q>4$]
    Throughout the proof, we will assume $p$ is sufficiently large depending on $q$, and all implied constants are allowed to depend on $q$. We may also assume $\sqp{x}\norm{a_x}_H^2=S(p^n)$ by normalization.

    We would like to reduce to the case that 
    \begin{equation}\label{eqs: q>four 1}
        \norm[\bigg]{\sqp{x}a_xe_{p^n}(tx)}_H\geq S(p^n)-O(p^{n-1})
    \end{equation}
    for some $t\in \Z/p^n$. To carry out the reduction, let $A$ be a large constant (depending only on $q$) chosen later, and suppose 
    \begin{equation*}
        \norm[\bigg]{\sqp{x}a_xe_{p^n}(tx)}_H\leq S(p^n)-Ap^{n-1}
    \end{equation*}
    for all $t\in \Z/p^n$. Then, by restriction estimates for $q=4$ (Proposition \ref{restriction for pn}), we have 
    \begin{equation*}
        \begin{split}
            \ap{t} \norm[\bigg]{\sqp{x}a_xe_{p^n}(tx)}_H^q&\leq (S(p^n)-Ap^{n-1})^{q-4} \ap{t} \norm[\bigg]{\sqp{x}a_xe_{p^n}(tx)}_H^4\\
            &\leq (S(p^n)-Ap^{n-1})^{q-4} (1+O(1/p))S(p^n)^4\\
            &=S(p^n)^q (1-Ap^{n-1}/S(p^n))^{q-4}(1+O(1/p)).
        \end{split}
    \end{equation*}
    Since $p^{n-1}/S(p^n)\gg 1/p$, choosing sufficiently large $A$ depending on $q$ will make $(1-Ap^{n-1}/S(p^n))^{q-4}(1+O(1/p))\leq 1$. Hence, we can reduce to that case that \eqref{eqs: q>four 1} holds for some $t\in \Z/p^n$. By multiplying $e_{p^n}(tx)$ to $a_x$, we may further assume that $t=0$ and hence
    \begin{equation}\label{eqs: q>four 2}
        \norm[\bigg]{\sqp{x}a_x}_H\geq S(p^n)-O(p^{n-1}).
    \end{equation}

    Now, let $c=\frac{1}{S(p^n)}\sqp{x}a_x$ be the average of $a_x$ and decompose $a_x=b_x+c$, where the vectors $b_x$ have mean zero. One has good lower and upper bounds for $\norm{c}_H$, that is
    \begin{equation*}
        1-O(1/p)\leq \norm{c}_H\leq 1
    \end{equation*}
    by \eqref{eqs: q>four 2} and Cauchy--Schwarz inequality respectively. Moreover, Pythagoras's theorem says that $\sqp{x}\norm{a_x}^2_H=\sqp{x}\norm{b_x}^2_H+S(p^n)\norm{c}^2_H$. Combining this and the lower bound for $\norm{c}_H$, we deduce a good $l^2$-bound for $b_x$, which is
    \begin{equation}\label{eqs: q>four 3}
        \sqp{x}\norm{b_x}^2_H\ll \frac{1}{p}S(p^n).
    \end{equation}

    Note that
    \begin{equation}\label{eqs: q>four 4}
        \begin{split}
            &\ap{t} \norm[\bigg]{\sqp{x}a_xe_{p^n}(tx)}_H^q\\
            &=\ap{t} \norm[\bigg]{\sqp{x}b_xe_{p^n}(tx)+\sqp{x}ce_{p^n}(tx)}_H^q\\
            &=\norm{S(p^n)c}^q_H+\sum_{\substack{t\,\mathrm{mod}\,p^n\\t\neq 0}} \norm[\bigg]{\sqp{x}b_xe_{p^n}(tx)+\sqp{x}ce_{p^n}(tx)}_H^q\\
            &\leq S(p^n)^q+\left\{\left(\sum_{\substack{t\,\mathrm{mod}\,p^n\\t\neq 0}} \norm[\bigg]{\sqp{x}b_xe_{p^n}(tx)}_H^q\right)^{1/q}+\left(\sum_{\substack{t\,\mathrm{mod}\,p^n\\t\neq 0}} \norm[\bigg]{\sqp{x}ce_{p^n}(tx)}_H^q\right)^{1/q}\right\}^q
        \end{split}
    \end{equation}
    where the last line follows from Minkowski inequality. We will estimate the sums involving $b_x$ and $c$ separately.

    For the sum involving $b_x$, we will use \eqref{eqs: q>four 3} and restriction estimates for $q=4$. We have
    \begin{equation}\label{eqs: q>four 5}
        \begin{split}
            \left(\sum_{\substack{t\,\mathrm{mod}\,p^n\\t\neq 0}} \norm[\bigg]{\sqp{x}b_xe_{p^n}(tx)}_H^q\right)^{1/q}&\leq \left(\sum_{\substack{t\,\mathrm{mod}\,p^n\\t\neq 0}} \norm[\bigg]{\sqp{x}b_xe_{p^n}(tx)}_H^4\right)^{1/4}\\
            &\leq C_4(p^n)S(p^n)^{1/2}\left(\sqp{x}\norm{b_x}^2_H\right)^{1/2}\\
            &\ll p^{n-1/2}.
        \end{split}
    \end{equation}

    For the sum involving $c$, we will use Gauss sum estimates (Proposition \ref{Gauss}). We have
    \begin{equation}\label{eqs: q>four 6}
        \sum_{\substack{t\,\mathrm{mod}\,p^n\\t\neq 0}} \norm[\bigg]{\sqp{x}ce_{p^n}(tx)}_H^q
        \leq \sum_{\substack{t\,\mathrm{mod}\,p^n\\t\neq 0}} \abs[\bigg]{\sqp{x}e_{p^n}(tx)}^q
        \ll \sum_{1\leq t< p^n} (t,p^n)^q p^{q/2},
    \end{equation}
    which is
    \begin{equation}\label{eqs: q>four 7}
    \begin{split}
        &\sum_{0\leq e\leq n-1}\sum_{\substack{1\leq t<p^n\\p^e||t}}p^{qe}p^{q/2}\\
        &\leq \sum_{0\leq e\leq n-1}p^{n-e}p^{qe+q/2}\\
        &=p^{n+q/2}\sum_{0\leq e\leq n-1}p^{(q-1)e}\\
        &\ll p^{n+q/2}p^{(n-1)(q-1)}\\
        &=p^{qn+1-q/2}.
    \end{split}        
    \end{equation}

    Combining \eqref{eqs: q>four 4}, \eqref{eqs: q>four 5}, \eqref{eqs: q>four 6} and \eqref{eqs: q>four 7}, we obtain
    \begin{equation*}
        \begin{split}
            \ap{t} \norm[\bigg]{\sqp{x}a_xe_{p^n}(tx)}_H^q&\leq S(p^n)^q+O(p^{qn-q/2}+p^{qn+1-q/2})\\
            &\leq S(p^n)^q(1+O(p^{-\frac{q}{2}+1})).
        \end{split}
    \end{equation*}
    This completes the proof.
\end{proof}

Now, we have all the ingredients needed for the proof of Theorem \ref{restriction for W}.

\begin{proof}[Proof of Theorem \ref{restriction for W}]
    Let $q>4$, by Proposition \ref{restriction for pn}, there exists a constant $M=M(q)$ such that
    \begin{equation*}
        C_q(p^n)\leq 1+O_q \left( \frac{1}{p^{\frac{q}{2}-1}} \right)
    \end{equation*}
    for all $p\geq M$ and $n\geq 1$.

    Let $W=\prod_i p_i^{n_i}$ be a modulus with $p_i\geq M$ for all $i$. By Lemma \ref{multiplicativity}, we have
    \begin{equation*}
        C_q(W)=\prod_i C_q(p_i^{n_i})\leq \prod_{p\geq M}(1+O_q(p^{-q/2+1}))\leq \exp\left(\sum_{p\geq M}O_q(p^{-q/2+1})\right)\ll_q 1
    \end{equation*}
    since $q/2-1>1$ when $q>4$.
\end{proof}

\section{Sumsets of Dense Subsets of $\mathbb{Z}/W$}\label{sec: sumsets}

With the transference principle, one can treat relatively dense subsets of squares in $\Z/W$ as if they are actually dense in $\Z/W$. Therefore, we would like a quantitative sumset estimates for dense subsets of $\Z/W$. Li and Pan \cite[Lemma 3.3]{LiPan} proved the desired sumset estimates for prime moduli, and we will generalize their results to general moduli.

\begin{thm}\label{sumsets}
    Let $0<\theta_1,\theta_2,...,\theta_s\leq 1$ with $\theta_1+\theta_2+...+\theta_s>1$. There exist positive constants $c(\theta_1,\theta_2,...,\theta_s)$ and $M(\theta_1,\theta_2,...,\theta_s)$ depending only on $\theta_1,\theta_2,...,\theta_s$ such that the following holds.

    Suppose $W=\prod_i p_i^{n_i}$ is an intger with $p_i\geq M(\theta_1,\theta_2,...,\theta_s)$ for all $i$, and $A_1,A_2,...,A_s$ are subsets of $\Z/W$ with $|A_j|\geq \theta_j W$ for all $1\leq j\leq s$. Then,
    \begin{equation*}
        1_{A_1}*1_{A_2}*...*1_{A_s}(x)\geq c(\theta_1,\theta_2,...,\theta_s)W^{s-1}
    \end{equation*}
    for all $x\in \Z/W$.
\end{thm}

Our proof of Theorem \ref{sumsets} is similar to Li and Pan's proof of \cite[Lemma 3.3]{LiPan}, just that we need the following more general Pollard--Kneser type result due to Green and Ruzsa \cite[Proposition 6.1]{GreenRuzsa}.

\begin{thm}[Green--Ruzsa]\label{Green--Ruzsa}
    Let $Z$ be a finite abelian group, and $D(Z)$ be the size of the largest proper subgroup of $Z$. Suppose $A,B$ are subsets of $Z$, then, for any $1\leq t\leq \min \{|A|,\,|B|\}$, we have
    \begin{equation*}
        \sum_{x\in Z} \min \{t,\,1_A*1_B(x)\}\geq t\min \{|Z|,\,|A|+|B|-t-D(Z)\}.
    \end{equation*}
\end{thm}

\begin{proof}[Proof of Theorem \ref{sumsets}]
    We will induct on $s$, the number of the subsets. The base case would be $s=2$. Note that
    \begin{equation*}
        \begin{split}
            1_{A_1}*1_{A_2}(x)&=\sum_{\substack{a_1\in A_1,a_2\in A_2\\a_1+a_2=x}}1\\
            &=|A_1\cap(x-A_2)|\\
            &=|A_1|+|x-A_2|-|A_1\cup(x-A_2)|\\
            &\geq \theta_1W+\theta_2W-W\\
            &=(\theta_1+\theta_2-1)W
        \end{split}
    \end{equation*}
    for all $x\in \Z/W$. This completes the proof for base case.

    Assume the Theorem \ref{sumsets} holds for $s-1$ subsets. Let $\epsilon$ be a small parameter and $M^{\prime}$ be a large parameter chosen later. We will assume that we are working with moduli $W=\prod_i p_i^{n_i}$ with $p_i\geq M^{\prime}$ for all $i$, in particular, $W$ will be sufficiently large. Let $A_1,A_2,...,A_s$ be subsets of $\Z/W$ with $|A_j|\geq \theta_j W$ for all $1\leq j\leq s$. Due to Theorem \ref{Green--Ruzsa}, we have
    \begin{equation*}
        \sum_{x\in \Z/W} \min \{\lceil \epsilon W\rceil,\,1_{A_1}*1_{A_2}(x)\}\geq \lceil \epsilon W\rceil\min \left\{W,\,\theta_1 W+\theta_2 W-\lceil \epsilon W\rceil-\frac{1}{\min_i p_i}W\right\}
    \end{equation*}
    by noting that $D(\Z/W)\leq \frac{1}{\min_i p_i}W$ since $D(\Z/W)$ divides $W$.

    Hence, we have
    \begin{equation*}
        \begin{split}
            &\sum_{\substack{x\in \Z/W\\1_{A_1}*1_{A_2}(x)>\lceil \epsilon^2 W\rceil}} \min \{\lceil \epsilon W\rceil,\,1_{A_1}*1_{A_2}(x)\}\\
            &\geq \lceil \epsilon W\rceil\min \left\{W,\,\theta_1 W+\theta_2 W-\lceil \epsilon W\rceil-\frac{1}{\min_i p_i}W\right\}-\lceil \epsilon^2 W\rceil W\\
            &\geq \min\left\{(\epsilon-2\epsilon^2)W^2,\,\epsilon\left(\theta_1+\theta_2-4\epsilon-\frac{1}{M^{\prime}}\right)W^2\right\}
        \end{split}
    \end{equation*}
    where the last line follows from assuming $M^{\prime}$ is sufficiently large so $W$ is sufficiently large and hence $\lceil \epsilon W\rceil\leq 2\epsilon W$ and $\lceil \epsilon^2 W\rceil\leq 2\epsilon^2 W$ for notational convenience.

    Denote $B=\{x\in\Z/W:1_{A_1}*1_{A_2}(x)>\lceil \epsilon^2 W\rceil\}$, then we conclude that
    \begin{equation}\label{eqs: sumsets 1}
        \begin{split}
            |B|&\geq \min\left\{\frac{\epsilon-2\epsilon^2}{\lceil \epsilon W\rceil}W^2,\,\frac{\epsilon\left(\theta_1+\theta_2-4\epsilon-\frac{1}{M^{\prime}}\right)}{\lceil \epsilon W\rceil}W^2\right\}\\
            &\geq \min\left\{\frac{\epsilon-2\epsilon^2}{\epsilon W+\epsilon^2 W}W^2,\,\frac{\epsilon\left(\theta_1+\theta_2-4\epsilon-\frac{1}{M^{\prime}}\right)}{\epsilon W+\epsilon^2 W}W^2\right\}\\
            &\quad\text{(Assuming $W$ is large enough so $\lceil\epsilon W\rceil\leq \epsilon W+\epsilon^2W$)}\\
            &=\min\left\{\frac{1-2\epsilon}{1+\epsilon}W,\,\frac{\theta_1+\theta_2-4\epsilon-\frac{1}{M^{\prime}}}{1+\epsilon}W\right\}.
        \end{split}
    \end{equation}

    Now, we will choose our parameters carefully. We will choose sufficiently small $\epsilon$ depending on $\theta_1,\theta_2,...,\theta_s$ and sufficiently large $M^{\prime}$ depending on $\epsilon,\theta_1,\theta_2,...,\theta_s$ so that we have the followings:
    \begin{itemize}
        \item $\lceil \epsilon W\rceil\leq \min\{\theta_1 W,\,\theta_2 W\}$ so that the assumption of Theorem \ref{Green--Ruzsa} is met.
        \item $\lceil \epsilon W\rceil\leq 2\epsilon W$, $\lceil \epsilon^2 W\rceil\leq 2\epsilon^2 W$ and $\lceil\epsilon W\rceil\leq \epsilon W+\epsilon^2W$ for notational convenience.
        \item $\frac{1-2\epsilon}{1+\epsilon}$ is close enough to $1$ and $\frac{\theta_1+\theta_2-4\epsilon-\frac{1}{M^{\prime}}}{1+\epsilon}$ is close enough to $\theta_1+\theta_2$ so that 
        \begin{equation*}
            \min\left\{\frac{1-2\epsilon}{1+\epsilon},\,\frac{\theta_1+\theta_2-4\epsilon-\frac{1}{M^{\prime}}}{1+\epsilon}\right\}+\theta_3+...+\theta_s>1.
        \end{equation*}
        \item Thus, by \eqref{eqs: sumsets 1}, we have $|B|\geq \eta W$ for some $\eta$ depending on $M^{\prime},\epsilon,\theta_1,\theta_2,...,\theta_s$ with $0<\eta\leq 1$ and $\eta+\theta_3+...+\theta_s>1$.
        \item By induction hypothesis, we may apply Theorem \ref{sumsets} for $s-1$ densities $\eta,\theta_3,...,\theta_s$ to obtain two constants $c(\eta,\theta_3,...,\theta_s)>0$ and $M(\eta,\theta_3,...,\theta_s)$. We claim that for $s$ densities $\theta_1,\theta_2,...,\theta_s$, we can pick $c(\theta_1,\theta_2,...,\theta_s)=\epsilon^2 c(\eta,\theta_3,...,\theta_s)$ and $M(\theta_1,\theta_2,...,\theta_s)=\max\{M^{\prime},\,M(\eta,\theta_3,...,\theta_s)\}$.
    \end{itemize}

    To finish the proof, note that for all $x\in \Z/W$, we have
    \begin{equation*}
        \begin{split}
            1_{A_1}*1_{A_2}*...*1_{A_s}(x)&\geq (1_{A_1}*1_{A_2})|_B*(1_{A_3}...*1_{A_s})(x)\\
            &\geq \epsilon^2 W (1_B*1_{A_3}...*1_{A_s})(x)\\
            &\geq \epsilon^2 W c(\eta,\theta_3,...,\theta_s)W^{s-2}\quad\quad\text{(By induction hypothesis)}\\
            &=\epsilon^2 c(\eta,\theta_3,...,\theta_s)W^{s-1}.
        \end{split}
    \end{equation*}
\end{proof}

\section{Transference Principle}\label{sec: transference principle}

In this section, we will use the transference principle pioneered by Green and Tao \cite{Green}, \cite{GreenTao} to prove Theorem \ref{main}. We will closely follow the arguments in Section 6 in \cite{Green}. As a remark, one can view the assumption of Theorem \ref{main} that $W$ does not have small prime divisors as an analog of the $W$-trick.

Let $s\geq 5$ be an integer and $\theta\in (1/s,1]$. Let $M$ be a large constant chosen later, and we will be working with moduli $W=\prod_i p_i^{n_i}$ with $\min_i p_i\geq M$ so that we can freely apply Gauss sum estimates, restriction theorem, quantitative sumset estimates, etc. Let $A$ be a subset of the squares in $\Z/W$ with $|A|\geq \theta S(W)$.

We define the following objects for the transference principle:
\begin{itemize}
    \item Let $\nu=\frac{1}{S(W)}1_{(\Z/W)^{(2)}}$ be a probability measure supported on the squares in $\Z/W$.
    \item Define the normalized indicator function of $A$ by setting $a=1_A \nu$. Note that the $l^1$-norm of $a$ is the relative density of $A$.
    \item Let $\delta,\epsilon>0$ be small parameters chosen later. Let
    \begin{equation*}
        R=\{\xi\in \Z/W:\abs{\widetilde{a}(\xi)}>\delta\}
    \end{equation*}
    be the large spectrum of $a$ and 
    \begin{equation*}
        B=\left\{x\in \Z/W:\norm[\bigg]{\frac{x\xi}{W}}_{\R/\Z}<\epsilon\;\text{for all $\xi\in R$}\right\}
    \end{equation*}
    be the Bohr neighborhood of $R$.
    \item Let $\beta=\frac{1}{|B|}1_B$ and $a'=a*\beta*\beta$. The function $a'$ can be thought of as the ``transferred'' version of $a$. Note that $a'$ and $a$ have the same $l^1$-norm.
\end{itemize}

First, we show that $A$ and the transferred version of $A$ are close in finding certain additive configurations. 

\begin{prop}\label{close in counting}
    For all $y\in \Z/W$, we have
    \begin{equation*}
        \sum_{\substack{x_1,x_2,...,x_s\in\Z/W\\y=x_1+x_2+...+x_s}}a(x_1)a(x_2)...a(x_s)=\sum_{\substack{x_1,x_2,...,x_s\in\Z/W\\y=x_1+x_2+...+x_s}}a'(x_1)a'(x_2)...a'(x_s)+O\left(\frac{1}{W}(\epsilon \delta^{-4.5}+\delta^{1/2})\right)
    \end{equation*}
\end{prop}

\begin{proof}
    By Fourier expansions, we obtain
    \begin{equation*}
        \begin{split}
            &\sum_{\substack{x_1,x_2,...,x_s\in\Z/W\\y=x_1+x_2+...+x_s}}a(x_1)a(x_2)...a(x_s)-\sum_{\substack{x_1,x_2,...,x_s\in\Z/W\\y=x_1+x_2+...+x_s}}a'(x_1)a'(x_2)...a'(x_s)\\
            &=\frac{1}{W}\sum_{\xi\in\Z/W}\widetilde{a}(\xi)^se_W(\xi y)-\frac{1}{W}\sum_{\xi\in\Z/W}\widetilde{a'}(\xi)^se_W(\xi y)\\
            &=\frac{1}{W}\sum_{\xi\in\Z/W}\widetilde{a}(\xi)^se_W(\xi y)(1-\widetilde{\beta}(\xi)^{2s}).
        \end{split}
    \end{equation*}
    We will divide the sum above into two cases, according to whether $\xi$ is in the large spectrum $R$.

    Note that by construction of $B$, we have $\widetilde{\beta}(\xi)=1+O(\epsilon)$ if $\xi\in R$. Also, by restriction estimates (Theorem \ref{restriction for W}), we obtain $\norm{\widetilde{a}}_q=O_q(1)$ when all the prime factors of $W$ are sufficiently large depending on $q$. In particular, we have $\norm{\widetilde{a}}_{4.5}=O(1)$ and hence $|R|=O(\delta^{-4.5})$. Therefore, we deduce
    \begin{equation*}
        \abs[\bigg]{\sum_{\xi\in R}\widetilde{a}(\xi)^se_W(\xi y)(1-\widetilde{\beta}(\xi)^{2s})}\ll \epsilon|R|\ll \epsilon \delta^{-4.5}.
    \end{equation*}
    On the other hand, we have
    \begin{equation*}
        \abs[\bigg]{\sum_{\xi\notin R}\widetilde{a}(\xi)^se_W(\xi y)(1-\widetilde{\beta}(\xi)^{2s})}\ll \delta^{1/2}\sum_{\xi\notin R}\abs{\widetilde{a}(\xi)}^{s-1/2}\ll \delta^{1/2}
    \end{equation*}
    by applying the restriction estimates for $q=s-1/2$ or $q=4.5$.
\end{proof}

The next proposition is about the pseudorandomness of the measure $\nu$.

\begin{prop}\label{pseudorandomness}
    For all $\xi\neq 0$, we have
    \begin{equation*}
        \abs{\widetilde{\nu}(\xi)}\ll \frac{1}{(\min_i p_i)^{1/2}}.
    \end{equation*}
    Recall that the minimum is taken among all prime factors of $W$.
\end{prop}

\begin{proof}
    First, we recall the multiplicativity of Gauss sums. Let $W=W_1W_2$ with $(W_1,W_2)=1$, then $\lambda_2W_1+\lambda_1W_2=1$ for some integers $\lambda_1,\lambda_2$ with $(\lambda_1,W_1)=(\lambda_2,W_2)=1$. We can decompose
    \begin{equation*}
        \sqw{x}e_W(\xi x)=\left(\sqwone{x_1}e_{W_1}(\lambda_1\xi x_1)\right)\left(\sqwtwo{x_2}e_{W_2}(\lambda_2\xi x_2)\right).
    \end{equation*}

    Now, if we write $W=\prod_i p_i^{n_i}$, then there exists a prime $p_j$ such that $p_j^{n_j}\nmid \xi$ since $\xi\neq 0$. By Gauss sum estimates (Proposition \ref{Gauss}) and the multiplicativity of Gauss sums, we have
    \begin{equation*}
        \begin{split}
            \abs{\widetilde{\nu}(\xi)}=\frac{1}{S(W)}\abs[\bigg]{\sqw{x}e_W(\xi x)}&\ll \frac{1}{S(W)} p_j^{n_j-1}p_j^{1/2} \prod_{\substack{i\\i\neq j}} S(p_i^{n_i})\\
            &=\frac{1}{S(p_j^{n_j})} p_j^{n_j-1/2}\\
            &\ll \frac{1}{p_j^{1/2}}\\
            &\leq \frac{1}{(\min_i p_i)^{1/2}}.
        \end{split}
    \end{equation*}
\end{proof}

Next, we show that $a'$ is set-like in the following proposition.

\begin{prop}\label{set-like}
    For all $x\in\Z/W$, we have
    \begin{equation*}
        a'(x)\leq \frac{1}{W}\left(1+O\left(\frac{1}{\epsilon^{O(\delta^{-4.5})}(\min_i p_i)^{1/2}}\right)\right)
    \end{equation*}
\end{prop}

\begin{proof}
    Note that
    \begin{equation*}
        \begin{split}
            a'(x)&\leq \nu*\beta*\beta(x)\\
            &=\frac{1}{W}\sum_{\xi\in\Z/W}(\nu*\beta*\beta)\widetilde{\;\;}(\xi)e_W(\xi x)\quad\quad\text{by Fourier expansion}\\
            &=\frac{1}{W}\sum_{\xi\in\Z/W}\widetilde{\nu}(\xi)\widetilde{\beta}(\xi)^2e_W(\xi x)\\
            &\leq\frac{1}{W}+\frac{1}{W}\sum_{\substack{\xi\in\Z/W\\\xi\neq 0}}\abs{\widetilde{\nu}(\xi)}\abs{\widetilde{\beta}(\xi)}^2
        \end{split}
    \end{equation*}

    By Proposition \ref{pseudorandomness} and Parseval's identity, we have
    \begin{equation*}
        \sum_{\substack{\xi\in\Z/W\\\xi\neq 0}}\abs{\widetilde{\nu}(\xi)}\abs{\widetilde{\beta}(\xi)}^2\ll \frac{W}{(\min_i p_i)^{1/2}|B|}.
    \end{equation*}

    The classical lower bound for the size of Bohr sets \cite{TaoVu}[Lemma 4.20] states that $|B|\geq \epsilon^{|R|} W$. To complete the proof, we observe that $|R|=O(\delta^{-4.5})$ since $\norm{\widetilde{a}}_{4.5}=O(1)$ due to restriction estimates (Theorem \ref{restriction for W}) for $q=4.5$.
\end{proof}

Now, we have all the ingredients to prove our main theorem.

\begin{proof}[Proof of Theorem \ref{main}]
    Let $s\geq 5$ be an integer and $\theta\in(1/s,1]$. We will choose a list of parameters carefully depending on $s,\theta$ as follow:
    \begin{itemize}
        \item Choose $\gamma>0$ sufficiently small such that $\frac{1-\gamma}{1+\gamma}\theta s>1$.
        \item Let $c'=c'(\frac{1-\gamma}{1+\gamma}\theta,\frac{1-\gamma}{1+\gamma}\theta,...,\frac{1-\gamma}{1+\gamma}\theta)>0$ and $M'=M'(\frac{1-\gamma}{1+\gamma}\theta,\frac{1-\gamma}{1+\gamma}\theta,...,\frac{1-\gamma}{1+\gamma}\theta)$ be the constants from the quantitative sumset estimates for dense subsets of cyclic groups (Theorem \ref{sumsets}) for the density $\frac{1-\gamma}{1+\gamma}\theta$.
        \item Let $M''$ be the constant from the restriction estimates (Theorem \ref{restriction for W}) for $q=4.5$.
        \item Choose $\delta$ and then $\epsilon$ sufficiently small such that the error term in Proposition \ref{close in counting} is less than $\frac{\gamma^s\theta^s c'}{2W}$.
        \item By Proposition \ref{set-like}, we can choose sufficiently large $M'''$ such that $\norm{a'}_{\infty}\leq \frac{1}{W}(1+\gamma)$ when $\min_i p_i\geq M'''$.
        \item For the proof of main theorem (Theorem \ref{main}), we claim that we can choose $c=\frac{\gamma^s\theta^s c'}{2}>0$ and $M=\max\{M',\,M'',\,M'''\}$.
    \end{itemize}

    Let 
    \begin{equation*}
        A'=\left\{x\in\Z/W:a'(x)\geq\frac{\gamma\theta}{W}\right\}
    \end{equation*}
    Since $\norm{a'}_{\infty}\leq \frac{1}{W}(1+\gamma)$ and $\norm{a'}_1\geq \theta$, by pigeonhole principle, we have $|A'|\geq \frac{1-\gamma}{1+\gamma}\theta W$.

    By the quantitative sumset estimates for dense subsets of $\Z/W$ (Theorem \ref{sumsets}), we have
    \begin{equation*}
        \begin{split}
            \underbrace{a'*a'*...*a'}_{s-\text{times}}(y)&\geq \left(\frac{\gamma\theta}{W}1_{A'}\right)*\left(\frac{\gamma\theta}{W}1_{A'}\right)*...*\left(\frac{\gamma\theta}{W}1_{A'}\right)(y)\\
            &=\left(\frac{\gamma\theta}{W}\right)^s (1_{A'}*1_{A'}*...*1_{A'})(y)\\
            &\geq \left(\frac{\gamma\theta}{W}\right)^s c'W^{s-1}\\
            &=\frac{\gamma^s\theta^s c'}{W}
        \end{split}
    \end{equation*}
    for all $y\in\Z/W$. Combining this and Proposition \ref{close in counting}, we conclude that
    \begin{equation*}
        \sum_{\substack{x_1,x_2,...,x_s\in\Z/W\\y=x_1+x_2+...+x_s}}a(x_1)a(x_2)...a(x_s)\geq \frac{\gamma^s\theta^s c'}{2W}
    \end{equation*}
    and complete the proof.
\end{proof}

\section{Counterexamples}\label{sec: counter examples}

In this section, we prove Theorem \ref{counter examples} to complement Theorem \ref{main}. By identifying the representatives of the elements of $\Z/p$ with the integers in $[1,p]$, we can speak of ``intervals'' in $\Z/p$. We will use squares in intervals in $\Z/p$ to construct subsets of squares with relative density close to $1/s$, but their $s$-fold sumsets are not entire $\Z/p$.

The next proposition, which states that the quadratic residues are equidistributed among intervals, is a classical result that can be proved in many different ways, say, by Polya-Vinogradov inequality. For self-containedness, we still include a short proof based on Gauss sums.

\begin{prop}\label{equidistribution}
    Let $0<\eta<1$, then
    \begin{equation*}
        \#\{x\in\Z:1\leq x\leq \eta p\;\;\text{and $x$ is a square mod $p$}\}=\eta\frac{p+1}{2}+O(\sqrt{p}\log p)
    \end{equation*}
    for all primes $p$.
\end{prop}

\begin{proof}
    By abuse of notation, we will denote the characteristic function of the interval $[1,\,\eta p]$ in $\Z/p$ by $1_{[1,\,\eta p]}$. By Fourier expansion, we have
    \begin{equation*}
        \begin{split}
            &\sum_{x=\square\,\mathrm{mod}\,p}1_{[1,\,\eta p]}(x)\\
            &=\frac{1}{p}\sum_{x=\square\,\mathrm{mod}\,p}\sum_{\xi\,\mathrm{mod}\,p}\widetilde{1}_{[1,\,\eta p]}(\xi)e_p(\xi x)\\
            &=\frac{\lfloor \eta p \rfloor}{p}\cdot\frac{p+1}{2}+\frac{1}{p}\sum_{\substack{\xi\,\mathrm{mod}\,p\\\xi\neq 0}}\widetilde{1}_{[1,\,\eta p]}(\xi)\sum_{x=\square\,\mathrm{mod}\,p}e_p(\xi x)\\
            &=\eta\frac{p+1}{2}+O(1)+\frac{1}{p}\sum_{\substack{\xi\,\mathrm{mod}\,p\\\xi\neq 0}}\widetilde{1}_{[1,\,\eta p]}(\xi)\sum_{x=\square\,\mathrm{mod}\,p}e_p(\xi x).
        \end{split}
    \end{equation*}

    To bound the error term, we calculate the Fourier coefficients
    \begin{equation*}
        \abs{\widetilde{1}_{[1,\,\eta p]}(\xi)}=\abs[\bigg]{\sum_{1\leq x\leq \lfloor \eta p \rfloor}e_p(-x\xi)}\ll \norm[\bigg]{\frac{\xi}{p}}_{\R/\Z}^{-1}
    \end{equation*}
    for all $\xi\neq 0$. Combining this and the Gauss sum estimates (Proposition \ref{Gauss}), we deduce that
    \begin{equation*}
        \begin{split}
            \abs[\bigg]{\frac{1}{p}\sum_{\substack{\xi\,\mathrm{mod}\,p\\\xi\neq 0}}\widetilde{1}_{[1,\,\eta p]}(\xi)\sum_{x=\square\,\mathrm{mod}\,p}e_p(\xi x)}&\ll \frac{1}{\sqrt{p}}\sum_{\substack{\xi\,\mathrm{mod}\,p\\\xi\neq 0}}\norm[\bigg]{\frac{\xi}{p}}_{\R/\Z}^{-1}\\
            &\ll \sqrt{p}\log p.
        \end{split}
    \end{equation*}
\end{proof}

Now, we will apply Proposition \ref{equidistribution} to prove Theorem \ref{counter examples}.

\begin{proof}[Proof of Theorem \ref{counter examples}]
    Let $s\geq 5$ be an integer, and $\tau\in [0,1/s)$. Choose $\eta$ with $\tau<\eta<1/s$, by Proposition \ref{equidistribution}, we have
    \begin{equation*}
        \#\{x\in\Z:1\leq x\leq \eta p\;\;\text{and $x$ is a square mod $p$}\}\geq\tau\frac{p+1}{2}
    \end{equation*}
    for sufficiently large $p$. However, the $s$-fold sumset of $\{x\in\Z:1\leq x\leq \eta p\;\;\text{and $x$ is a square mod $p$}\}$ cannot be $\Z/p$ since $\eta s<1$.
\end{proof}

\section{A Combinatorial Lemma}\label{sec: combinatorial lemma}

One of the key ingredients in Shao's \cite{Shao} proof of the sharp density version of Vinogradov three primes theorem is a combinatorial inequality. Subsequently, Lacey, Mousavi, Rahimi and Vempati \cite{lacey2024densitytheoremhigherorder} developed a variant of Shao's inequality to generalize the density version of Vinogradov theorem to sum of more primes. In this section, we will use variants of Shao's and Lacey, Mousavi, Rahimi and Vempati's ideas to prove a combinatorial inequality that suits our purpose. 

The combinatorial inequality is inductive in nature, but our base cases are more complicated, hence we handled them by linear programming with computer assistance. As a remark, Shao also used linear programming techniques in his proof, but the linear programming problems we dealt with are different due to different formulations of the combinatorial inequalities.

\begin{lem}\label{combinatorial lemma}
    Let $s\geq 6, n\geq 3$ or $s=5, n\geq 5$. For each integer $1\leq j\leq s$, let $\{a_{i,j}\}_{0\leq i<n}$ be a decreasing sequence in interval $[0,1]$. Suppose 
    \begin{equation*}
        \text{$\frac{1}{ns}\sum_{\substack{0\leq i<n\\1\leq j\leq s}}a_{i,j}>d_s$ and $a_{0,j}\neq 0$ for all $1\leq j\leq s$.}
    \end{equation*}
    Then, there exist indices $0\leq i_1,i_2,...,i_s<n$ with $i_1+i_2+...+i_s\geq 2n$ such that
    \begin{equation*}
        \text{$\frac{1}{s}\sum_{1\leq j\leq s}a_{i_j,j}>d_s$ and $a_{i_j,j}\neq 0$ for all $1\leq j\leq s$.}
    \end{equation*}
\end{lem}

The proof of Lemma \ref{combinatorial lemma} is a bit long, so we separate it into inductive step and base cases. Along the proof, we will see explanations for the somewhat awkward definition of $d_s$. The reason why we choose $d_s$ to be greater than or equal to $(s+13)/4s$ is to make the inductive step work.

\begin{proof}[Proof of Lemma \ref{combinatorial lemma} assuming the base cases]
    We will prove Lemma \ref{combinatorial lemma} by induction on $n$. Let $s\geq 6, n\geq 4$ or $s=5, n\geq 6$. Suppose the lemma holds for $n-1$.

    There are at least two nonzero elements among $a_{1,j}$ where $1\leq j\leq s$, otherwise the total average is not large enough, that is
    \begin{equation*}
        \frac{1}{ns}\sum_{\substack{0\leq i<n\\1\leq j\leq s}}a_{i,j}\leq \frac{1}{ns}(n+(s-1))\leq \frac{s+3}{4s}
    \end{equation*}
    since $n\geq 4$. However, the total average is assumed to be greater than $d_s$, which is at least $(s+13)/(4s)$. Thus, by relabeling if needed, we may assume $a_{1,1},a_{1,2}\neq 0$.

    Suppose we have
    \begin{equation*}
        \frac{1}{s}(a_{0,1}+a_{0,2}+a_{i_3,3}+a_{i_4,4}+...+a_{i_s,s})\leq d_s\;\text{for some indices}\;0\leq i_3,i_4,...,i_s<n.
    \end{equation*}
    By monotonicity of $a_{i,j}$ in $i$, we may further assume
    \begin{equation*}
        \frac{1}{s}(a_{0,1}+a_{0,2}+a_{i_3,3}+a_{i_4,4}+...+a_{i_s,s})\leq d_s\;\text{for some indices}\;1\leq i_3,i_4,...,i_s<n.
    \end{equation*}
    Remove the elements $a_{0,1},a_{0,2},a_{i_3,3},a_{i_4,4},...,a_{i_s,s}$ from the sequences $\{a_{i,j}\}_{0\leq i<n,1\leq j\leq s}$ and denote the new sequences after removal by $\{b_{i,j}\}_{0\leq i<n-1,1\leq j\leq s}$. Note that by construction, we have
    \begin{equation}\label{eqs: combinatorial lemma 1}
        \begin{split}
            b_{i,j}=a_{i+1,j}&\;\;\text{if}\;j=1,2\\
            b_{i,j}\leq a_{i,j}&\;\;\text{if}\;j\geq 3.
        \end{split}
    \end{equation}

    Observe that the total average of $b_{i,j}$ is greater than or equal to the total average of $a_{i,j}$ since the average of the removed elements is less than the total average of $a_{i,j}$, that is
    \begin{equation*}
        \frac{1}{(n-1)s}\sum_{\substack{0\leq i<n-1\\1\leq j\leq s}}b_{i,j}\geq \frac{1}{ns}\sum_{\substack{0\leq i<n\\1\leq j\leq s}}a_{i,j}>d_s.
    \end{equation*}
    Also, the elements $b_{0,j}$ are nonzero for $1\leq j\leq s$ because $b_{0,j}=a_{1,j}\neq 0$ for $j=1,2$ and $b_{0,j}=a_{0,j}\neq 0$ for $j\geq 3$. In short, the sequences $\{b_{i,j}\}_{0\leq i<n-1,1\leq j\leq s}$ satisfy all assumptions in Lemma \ref{combinatorial lemma} for $n-1$. By induction hypothesis, there exist indices $i_1,i_2,...,i_s$ with $i_1+i_2+...+i_s\geq 2(n-1)$ such that
    \begin{equation*}
        \text{$\frac{1}{s}\sum_{1\leq j\leq s}b_{i_j,j}>d_s$ and $b_{i_j,j}\neq 0$ for all $1\leq j\leq s$.}
    \end{equation*}

    Note that by equations \eqref{eqs: combinatorial lemma 1}, we conclude that
    \begin{equation*}
        \frac{1}{s}(a_{i_1+1,1}+a_{i_2+1,2}+a_{i_3,3}+...+a_{i_s,s})>d_s.
    \end{equation*}
    Furthermore, the sum of indices is $(i_1+1)+(i_2+1)+i_3+...+i_s\geq 2(n-1)+2=2n$, hence the elements $a_{i_1+1,1},a_{i_2+1,2},a_{i_3,3},...,a_{i_s,s}$ are the desired collection of the elements for Lemma \ref{combinatorial lemma}.

    Now, we consider the other possibility, namely
    \begin{equation*}
        \frac{1}{s}(a_{0,1}+a_{0,2}+a_{i_3,3}+a_{i_4,4}+...+a_{i_s,s})> d_s\;\text{for all indices}\;0\leq i_3,i_4,...,i_s<n.
    \end{equation*}
    For each $1\leq j\leq s$, let $t_j$ be the number of nonzero elements among $\{a_{i,j}\}_{0\leq i<n}$. In particular, we have
    \begin{equation*}
        \frac{1}{s}(a_{0,1}+a_{0,2}+a_{t_3-1,3}+a_{t_4-1,4}+...+a_{t_s-1,s})> d_s.
    \end{equation*}
    It remains to check that the sum of indices is large enough. Since $a_{i,j}$ are all bounded by $1$, we have
    \begin{equation*}
        2n+(t_3+t_4+...+t_s)\geq \sum_{\substack{0\leq i<n\\1\leq j\leq s}}a_{i,j}>nsd_s
    \end{equation*}

    Hence, we deduce that
    \begin{equation*}
        \begin{split}
            (t_3-1)+(t_4-1)+...+(t_s-1)&>n(sd_s-2)-(s-2)\\
            &\geq n\left(\frac{s+13}{4}-2\right)-(s-2)\;\;\left(\text{since $d_s\geq \frac{s+13}{4s}$ for all $s$}\right)\\
            &=n\left(\frac{s+5}{4}\right)-(s-2).
        \end{split}
    \end{equation*}
    Note that $n(s+5)/4-(s-2)\geq 2n-1$ when $n\geq 4$, thus we have
    \begin{equation*}
        (t_3-1)+(t_4-1)+...+(t_s-1)>2n-1.
    \end{equation*}
    Since $(t_3-1)+(t_4-1)+...+(t_s-1)$ is an integer, we conclude that
    \begin{equation*}
        (t_3-1)+(t_4-1)+...+(t_s-1)\geq 2n
    \end{equation*}
    and this completes the inductive step.
\end{proof}

For the base cases of Lemma \ref{combinatorial lemma} when $5\leq s\leq 11$, we will handle them via linear programming with computer assistance. For certain linear programming problems to have expected optimal values, we choose $d_s=5/6$ instead of $(s+13)/4s$ when $s=6$. This is also the reason why the base case for $s=5$ corresponds to $n=5$ as opposed to $n=3$.

\begin{proof}[Proof of base cases of Lemma \ref{combinatorial lemma} when $5\leq s\leq 11$]
    The crucial observation is for any fixed $s,n$, the Lemma \ref{combinatorial lemma} can be approached by linear programming. Let $s,n$ be fixed, for any integral indices $1\leq t_1,t_2,...,t_s\leq n$, we can define the following linear programming problem \textbf{LP}$(t_1,t_2,...,t_2)$.\\

    \noindent\textbf{LP}$(t_1,t_2,...,t_s)$:
    
    \textbf{Variables}: $a_{i,j}$ for $0\leq i<n, 1\leq j\leq s$
    
    \textbf{Constraints}:
    \begin{enumerate}
        \item For all $0\leq i<n, 1\leq j\leq s$, we require that $0\leq a_{i,j}\leq 1$.
        \item For all $1\leq j\leq s$ and $i\geq t_j$, we require that $a_{i,j}=0$.
        \item For all $0\leq i<n-1, 1\leq j\leq s$, we require that $a_{i,j}\geq a_{i+1,j}$.
        \item For all indices $0\leq i_1,i_2,...,i_s<n$ with $i_1+i_2+....+i_s\geq 2n$ and $i_1<t_1,i_2<t_2,...,i_s<t_s$, we require that
        \begin{equation*}
            \frac{1}{s}(a_{i_1,1}+a_{i_2,2}+...+a_{i_s,s})\leq d_s.
        \end{equation*}
    \end{enumerate}
    
    \textbf{Objective}: Maximize
    \begin{equation*}
        \frac{1}{ns}\sum_{\substack{0\leq i<n\\1\leq j\leq s}} a_{i,j}
    \end{equation*}
    
    We claim that if the maxima for linear programming problems \textbf{LP}$(t_1,t_2,...,t_s)$ are upper bounded by $d_s$ for all indices $1\leq t_1,t_2,...,t_s\leq n$, then Lemma \ref{combinatorial lemma} is true for the parameters $s,n$. 
    
    To see why the claim is true, we suppose the contrary, that is, the maxima for linear programming problems \textbf{LP}$(t_1,t_2,...,t_s)$ are upper bounded by $d_s$ for all indices $1\leq t_1,t_2,...,t_s\leq n$, but Lemma \ref{combinatorial lemma} is false for the parameters $s,n$. Let $\{a_{i,j}\}_{0\leq i<n, 1\leq j\leq s}$ be the hypothetical counterexample and $T_j$ be the number of nonzero elements among $\{a_{i,j}\}_{0\leq i<n}$ for every $1\leq j\leq s$. Note that the numbers $\{a_{i,j}\}_{0\leq i<n, 1\leq j\leq s}$ satisfy all constraints of the linear programming problem \textbf{LP}$(T_1,T_2,...,T_s)$, hence the total average of $\{a_{i,j}\}_{0\leq i<n, 1\leq j\leq s}$ is upper bounded by the maximum for this linear programming problem, which is less than or equal to $d_s$. This leads to a contradiction.

    Now, the coefficients involved in the linear programming problems \textbf{LP}$(t_1,t_2,...,t_s)$ are rational numbers, hence the maxima will be rational numbers. Therefore, these linear programming problems can be solved by computers without numerical errors in finitely many steps. We used a computer program to verify that in the cases
    \begin{equation*}
        (s,n)=(5,5),(6,3),(7,3),(8,3),(9,3),(10,3),(11,3),
    \end{equation*}
    the maxima for linear programming problems \textbf{LP}$(t_1,t_2,...,t_s)$ are indeed upper bounded by $d_s$. For those $(s,n)$, there are $126,28,36,45,55,66,78$ linear programming problems for each $(s,n)$ respectively. The codes we used were initially generated by ChatGPT and then modified by the author. The codes run in Sage and took about one hour to run on a personal computer.
\end{proof}

For the base cases of Lemma \ref{combinatorial lemma} when $s\geq 12$, the linear programming problems become quite demanding for a personal computer. Fortunately, in this range, we have enough variables to handle them by hand. The reason why we choose $d_s=1/2$ but not $(s+13)/4s$ when $s\geq 14$ is because it is impossible to prove Theorem \ref{main for primes} for threshold lower than $1/2$, since the set $\{p\in \mathcal{P}:p\equiv 1,4\;\mathrm{mod}\;5\}$ has relative lower density $1/2$, but the squares of its elements are always congruent to $1\;\mathrm{mod}\;5$.

We seperate the proof of base cases of Lemma \ref{combinatorial lemma} when $s\geq 12$ into the following lemma.

\begin{lem}\label{lemma for base cases}
    Let $s\geq 12$ be an integer, and $D\geq \frac{1}{2}$ be a real number. For each integer $1\leq j\leq s$, let $\{a_{i,j}\}_{0\leq i<3}$ be a decreasing sequence in interval $[0,1]$. Suppose 
    \begin{equation*}
        \text{$\frac{1}{3s}\sum_{\substack{0\leq i<3\\1\leq j\leq s}}a_{i,j}>D$ and $a_{0,j}\neq 0$ for all $1\leq j\leq s$.}
    \end{equation*}
    Then, there exist indices $0\leq i_1,i_2,...,i_s<3$ with $i_1+i_2+...+i_s\geq 6$ such that
    \begin{equation*}
        \text{$\frac{1}{s}\sum_{1\leq j\leq s}a_{i_j,j}>D$ and $a_{i_j,j}\neq 0$ for all $1\leq j\leq s$.}
    \end{equation*}
\end{lem}

\begin{proof}[Proof of Lemma \ref{lemma for base cases} and bases cases of Lemma \ref{combinatorial lemma} when $s\geq 12$]
    Let $R$ be the number of nonzero elements among $a_{2,j}$ where $1\leq j\leq s$. We will prove Lemma \ref{lemma for base cases} via cases-by-cases analysis based on the values of $R$. Note that the number of nonzero elements among $a_{i,j}$ where $i=1,2$ and $1\leq j\leq s$ is $\geq 7$, because this number is greater than $s/2$.\\

    \noindent\underline{\textbf{Case 0:}} $R=0$

    There are at least $7$ nonzero elements among $a_{i,j}$ where $i=1,2$ and $1\leq j\leq s$. By relabeling, we may assume $a_{1,j}\neq 0$ for $1\leq j\leq 7$. Suppose Lemma \ref{lemma for base cases} is false, then
    \begin{align*}
        &(a_{1,1}+a_{1,2}+a_{1,3}+a_{1,4}+a_{1,5}+a_{1,6})+(a_{0,7}+a_{0,8}+...+a_{0,s})\leq sD,\\
        &a_{0,1}+(a_{1,2}+a_{1,3}+a_{1,4}+a_{1,5}+a_{1,6}+a_{1,7})+(a_{0,8}+a_{0,9}+...+a_{0,s})\leq sD.
    \end{align*}
    Summing the inequalities above and using the facts that $a_{i,j}$ is decreasing in $i$ and $R=0$, we have
    \begin{equation*}
        \sum_{\substack{0\leq i<3\\1\leq j\leq s}}a_{i,j}-(a_{0,2}+a_{0,3}+a_{0,4}+a_{0,5}+a_{0,6})\leq 2sD.
    \end{equation*}
    This implies $3sD-5<2sD$, which leads to a contradiction.\\

    The strategies for the remaining cases are similar to Case $0$.\\

    \noindent\underline{\textbf{Case 1:}} $R=1$

    There are at least $7$ nonzero elements among $a_{i,j}$ where $i=1,2$ and $1\leq j\leq s$. By relabeling, we may assume $a_{1,j}\neq 0$ for $1\leq j\leq 6$ and $a_{2,1}\neq 0$. Suppose Lemma \ref{lemma for base cases} is false, then
    \begin{align*}
        &a_{2,1}+(a_{1,2}+a_{1,3}+a_{1,4}+a_{1,5})+(a_{0,6}+a_{0,7}+...+a_{0,s})\leq sD,\\
        &(a_{1,1}+a_{1,2}+a_{1,3}+a_{1,4}+a_{1,5}+a_{1,6})+(a_{0,7}+a_{0,8}+...+a_{0,s})\leq sD.
    \end{align*}
    Summing the inequalities above and using the facts that $a_{i,j}$ is decreasing in $i$ and $R=1$, we have
    \begin{equation*}
        \sum_{\substack{0\leq i<3\\1\leq j\leq s}}a_{i,j}-(a_{0,1}+a_{0,2}+a_{0,3}+a_{0,4}+a_{0,5})\leq 2sD.
    \end{equation*}
    This implies $3sD-5<2sD$, which leads to a contradiction.\\

    \noindent\underline{\textbf{Case 2:}} $R=2$

    There are at least $7$ nonzero elements among $a_{i,j}$ where $i=1,2$ and $1\leq j\leq s$. By relabeling, we may assume $a_{1,j}\neq 0$ for $1\leq j\leq 5$ and $a_{2,1},a_{2,2}\neq 0$. Suppose Lemma \ref{lemma for base cases} is false, then
    \begin{align*}
        &(a_{2,1}+a_{2,2})+(a_{1,3}+a_{1,4})+(a_{0,5}+a_{0,6}+...+a_{0,s})\leq sD,\\
        &a_{2,1}+(a_{1,2}+a_{1,3}+a_{1,4}+a_{1,5})+(a_{0,6}+a_{0,7}+...+a_{0,s})\leq sD.
    \end{align*}
    Summing the inequalities above and using the facts that $a_{i,j}$ is decreasing in $i$ and $R=2$, we have
    \begin{equation*}
        \sum_{\substack{0\leq i<3\\1\leq j\leq s}}a_{i,j}-(a_{0,1}+a_{1,1}+a_{0,2}+a_{0,3}+a_{0,4})\leq 2sD.
    \end{equation*}
    This implies $3sD-5<2sD$, which leads to a contradiction.\\

    \noindent\underline{\textbf{Case 3:}} $R=3$

    There are at least $7$ nonzero elements among $a_{i,j}$ where $i=1,2$ and $1\leq j\leq s$. By relabeling, we may assume $a_{1,j}\neq 0$ for $1\leq j\leq 4$ and $a_{2,1},a_{2,2},a_{2,3}\neq 0$. Suppose Lemma \ref{lemma for base cases} is false, then
    \begin{align*}
        &(a_{2,1}+a_{2,2}+a_{2,3})+(a_{0,4}+a_{0,5}+a_{0,6}+a_{0,7}+...+a_{0,s})\leq sD,\\
        &(a_{2,1}+a_{2,2})+(a_{1,3}+a_{1,4})+(a_{0,5}+a_{0,6}+...+a_{0,s})\leq sD.
    \end{align*}
    Summing the inequalities above and using the facts that $a_{i,j}$ is decreasing in $i$ and $R=3$, we have
    \begin{equation*}
        \sum_{\substack{0\leq i<3\\1\leq j\leq s}}a_{i,j}-(a_{0,1}+a_{1,1}+a_{0,2}+a_{1,2}+a_{0,3})\leq 2sD.
    \end{equation*}
    This implies $3sD-5<2sD$, which leads to a contradiction.\\

    \noindent\underline{\textbf{Case 4:}} $R=4$

    By relabeling, we may assume $a_{i,j}\neq 0$ for $1\leq i\leq 2,1\leq j\leq 4$. Suppose Lemma \ref{lemma for base cases} is false, then
    \begin{align*}
        &(a_{2,1}+a_{2,2})+(a_{1,3}+a_{1,4})+(a_{0,5}+a_{0,6}+...+a_{0,s})\leq sD,\\
        &(a_{1,1}+a_{1,2})+(a_{2,3}+a_{2,4})+(a_{0,5}+a_{0,6}+...+a_{0,s})\leq sD.
    \end{align*}
    Summing the inequalities above and using the facts that $a_{i,j}$ is decreasing in $i$ and $R=4$, we have
    \begin{equation*}
        \sum_{\substack{0\leq i<3\\1\leq j\leq s}}a_{i,j}-(a_{0,1}+a_{0,2}+a_{0,3}+a_{0,4})\leq 2sD.
    \end{equation*}
    This implies $3sD-4<2sD$, which leads to a contradiction.\\

    \noindent\underline{\textbf{Case 5:}} $R=5$

    By relabeling, we may assume $a_{i,j}\neq 0$ for $1\leq i\leq 2,1\leq j\leq 5$. Suppose Lemma \ref{lemma for base cases} is false, then
    \begin{align*}
        &(a_{2,1}+a_{2,2})+(a_{1,3}+a_{1,4}+a_{1,5})+(a_{0,6}+a_{0,7}+...+a_{0,s})\leq sD,\\
        &(a_{1,1}+a_{1,2})+(a_{2,3}+a_{2,4}+a_{2,5})+(a_{0,6}+a_{0,7}+...+a_{0,s})\leq sD.
    \end{align*}
    Summing the inequalities above and using the facts that $a_{i,j}$ is decreasing in $i$ and $R=5$, we have
    \begin{equation*}
        \sum_{\substack{0\leq i<3\\1\leq j\leq s}}a_{i,j}-(a_{0,1}+a_{0,2}+a_{0,3}+a_{0,4}+a_{0,5})\leq 2sD.
    \end{equation*}
    This implies $3sD-5<2sD$, which leads to a contradiction.\\

    \noindent\underline{\textbf{Case 6:}} $R\geq 6$

    By relabeling, we may assume $a_{i,j}\neq 0$ for $1\leq i\leq 2,1\leq j\leq 6$. Suppose Lemma \ref{lemma for base cases} is false, then
    \begin{align*}
        &(a_{0,1}+a_{0,2}+a_{0,3})+(a_{2,4}+a_{2,5}+a_{2,6})+(a_{0,7}+a_{0,8}+...+a_{0,s})\leq sD,\\
        &(a_{1,1}+a_{1,2}+a_{1,3}+a_{1,4}+a_{1,5}+a_{1,6})+(a_{0,7}+a_{0,8}+...+a_{0,s})\leq sD,\\
        &(a_{2,1}+a_{2,2}+a_{2,3})+(a_{0,4}+a_{0,5}+a_{0,6})+(a_{0,7}+a_{0,8}+...+a_{0,s})\leq sD.
    \end{align*}
    Summing the inequalities above and using the facts that $a_{i,j}$ is decreasing in $i$, we have
    \begin{equation*}
        \sum_{\substack{0\leq i<3\\1\leq j\leq s}}a_{i,j}\leq 3sD.
    \end{equation*}
    This implies $3sD<3sD$, which leads to a contradiction.\\
\end{proof}

Now, we will apply Lemma \ref{combinatorial lemma} to prove Theorem \ref{main for almost all moduli}. The method is similar to that of Shao \cite[Proposition 3.1]{Shao} and Li and Pan \cite[Theorem 1.2]{LiPan}.

\begin{proof}[Proof of Theorem \ref{main for almost all moduli}]
    We will prove Theorem \ref{main for almost all moduli} by induction on the number of prime divisors of $W$. The base case is $W$ being a single prime $p$. Note that by coprime assumptions in Theorem \ref{main for almost all moduli}, we have $p\geq 7$ when $s\geq 6$ and $p\geq 11$ when $s=5$.

    For every $1\leq j\leq s$, write the image of $f_j$ as a decreasing seqeunce $a_{0,j}\geq a_{1,j}\geq a_{2,j}\geq ...\geq a_{n-1,j}$ in interval $[0,1]$, where $n=(p-1)/2$. Note that by lower bounds for $p$, we have $n\geq 3$ when $s\geq 6$ and $n\geq 5$ when $s=5$. Moreover, we have lower bound for the total average of $a_{i,j}$, that is
    \begin{equation*}
        \frac{1}{ns}\sum_{\substack{0\leq i<n\\1\leq j\leq s}}a_{i,j}=\frac{1}{s}(\E[f_1]+\E[f_2]+...+\E[f_s])>d_s.
    \end{equation*}
    Also, the leading terms $a_{0,j}$ are nonzero since $\E[f_j]\neq 0$ for $1\leq j\leq s$. Therefore, we can apply Lemma \ref{combinatorial lemma} and conclude that there exist indices $i_1,i_2,...,i_s$ with $i_1+i_2+...+i_s\geq 2n$ such that
    \begin{equation*}
        \frac{1}{s}(a_{i_1,1}+a_{i_2,2}+...+a_{i_s,s})>d_s\;\text{and}\;a_{i_j,j}\neq 0\;\text{for all}\;1\leq j\leq s.
    \end{equation*}

    Let $I_j=\{x\in(\Z/p)^{\times(2)}:f_j(x)\geq a_{i_j,j}\}$ for $1\leq j\leq s$. The size of $I_j$ is at least $i_j+1$ for $1\leq j\leq s$. By the Cauchy--Davenport inequality, we have
    \begin{equation*}
        \begin{split}
            |I_1+I_2+...+I_s|&\geq \min\{p,\;|I_1|+|I_2|+...+|I_s|-(s-1)\}\\
            &\geq \min\{p,\;i_1+i_2+...+i_s+1\}\\
            &\geq \min\{p,\;2n+1\}\\
            &=p
        \end{split}
    \end{equation*}
    hence the sumset $I_1+I_2+...+I_s$ is $\Z/p$. Thus, for all $y\in\Z/p$, there exist $x_j\in I_j$ for $1\leq j\leq s$ such that $y=x_1+x_2+...+x_s$. By constructions, we have $f_j(x_j)\geq a_{i_j,j}>0$ for all $1\leq j\leq s$ and
    \begin{equation*}
        \frac{1}{s}(f_1(x_1)+f_2(x_2)+...+f_s(x_s))\geq \frac{1}{s}(a_{i_1,1}+a_{i_2,2}+...+a_{i_s,s})>d_s.
    \end{equation*}
    This completes the proof of base case.

    Now, for general modulus $W$, suppose Theorem \ref{main for almost all moduli} is true for all moduli whose numbers of prime divisors are strictly less than the number of prime divisors of $W$. Write $W=pW'$ for some prime $p$ and square free integer $W'$, where $W'$ has one less prime divisors than $W$. By Chinese remainder theorem, we can identify $(\Z/W)^{\times(2)}$ with $(\Z/p)^{\times(2)}\times(\Z/W')^{\times(2)}$. The elements in $(\Z/W)^{\times(2)}$ can be conveniently written as $(x,x')$ where $x\in (\Z/p)^{\times(2)},x'\in (\Z/W')^{\times(2)}$.

    For each $1\leq j\leq s$, define $g_j:(\Z/p)^{\times(2)}\longrightarrow[0,1]$ by setting
    \begin{equation*}
        g_j(x)=\E_{x'\in(\Z/W')^{\times(2)}}f_j(x,x')
    \end{equation*}
    for all $x\in(\Z/p)^{\times(2)}$. Note that we have $\E[g_j]=\E[f_j]\neq 0$ for all $1\leq j\leq s$ and
    \begin{equation*}
        \frac{1}{s}(\E[g_1]+\E[g_2]+...+\E[g_s])=\frac{1}{s}(\E[f_1]+\E[f_2]+...+\E[f_s])>d_s.
    \end{equation*}
    Since we know that the base case is true, for all $(y,y')\in \Z/W$, we have $y=x_1+x_2+...+x_s$ for some $x_1,x_2,...,x_s\in (\Z/p)^{\times(2)}$ such that
    \begin{equation*}
        \text{$\frac{1}{s}(g_1(x_1)+g_2(x_2)+...+g_s(x_s))>d_s$ and $g_j(x_j)\neq 0$ for all $1\leq j \leq s$.}
    \end{equation*}

    For each $1\leq j\leq s$, define $h_j:(\Z/W')^{\times(2)}\longrightarrow[0,1]$ by setting
    \begin{equation*}
        h_j(x')=f_j(x_j,x')
    \end{equation*}
    for all $x'\in (\Z/W')^{\times(2)}$. Note that we have $\E[h_j]=g_j(x_j)\neq 0$ for all $1\leq j\leq s$ and
    \begin{equation*}
        \frac{1}{s}(\E[h_1]+\E[h_2]+...+\E[h_s])=\frac{1}{s}(g_1(x_1)+g_2(x_2)+...+g_s(x_s))>d_s.
    \end{equation*}
    By induction hypothesis, we have $y'=x_1'+x_2'+...+x_s'$ for some $x_1',x_2',...,x_s'\in (\Z/W')^{\times(2)}$ such that
    \begin{equation*}
        \text{$\frac{1}{s}(h_1(x_1')+h_2(x_2')+...+h_s(x_s'))>d_s$ and $h_j(x_j')\neq 0$ for all $1\leq j \leq s$.}
    \end{equation*}
    which is the same as
    \begin{equation*}
        \text{$\frac{1}{s}(f_1(x_1,x_1')+f_2(x_2,x_2')+...+f_s(x_s,x_s'))>d_s$ and $f_j(x_j,x_j')\neq 0$ for all $1\leq j \leq s$.}
    \end{equation*}
    Hence, $(y,y')=(x_1,x_1')+(x_2,x_2')+...+(x_s,x_s')$ is the desired decomposition.
\end{proof}

\section{Extensions to Small Moduli}\label{sec: small moduli}

In this section, we will extend Theorem \ref{main for almost all moduli} to cover the small moduli. Extensions to small moduli cause the density threshold to go up from $d_s$ to $D_s$.

\begin{prop}\label{main for small moduli}
    Let $s\geq 6, m=5$ or $s=5, m=35$. Let $f:(\Z/m)^{\times(2)}\longrightarrow[0,1]$ be a function with $\E[f]>D_s$. For all $y\in \Z/m$, we have $y=x_1+x_2+...+x_s$ for some $x_1,x_2,...,x_s\in (\Z/m)^{\times(2)}$ such that
    \begin{equation*}
        \text{$\frac{1}{s}(f(x_1)+f(x_2)+...+f(x_s))>d_s$ and $f(x_j)\neq 0$ for all $1\leq j \leq s$.}
    \end{equation*}
\end{prop}

\begin{proof}
    We will divide into four different cases according to the values of $s$.\\

    \noindent\underline{When $s\geq 8$:}

    There are two elements in $(\Z/5)^{\times(2)}$, which are $1$ and $4$. Note that $f(1)$ and $f(4)$ are nonzero since $\E[f]>D_s\geq 1/2$.

    Observe that for all $y\in \Z/5$, there are two ways to write $y$ as a sum of $1$'s and $4$'s, so that the number of $1$'s is not less than the number of $4$'s in one way, and the number of $4$'s is not less than the number of $1$'s in the other way. This is because we can consider
    \begin{equation*}
        \underbrace{1+1+...+1}_{r\;\text{times}}+\underbrace{4+4+...+4}_{s-r\;\text{times}}=r+4(s-r)=4s-3r
    \end{equation*}
    for $0\leq r\leq 4$ and $s-4\leq r\leq s$. Note that $4s-3r$ ranges through all congruence classes $\mathrm{mod}\;5$ when $r$ ranges through all congruence classes $\mathrm{mod}\;5$.

    Thus, for all $y\in \Z/5$, we can pick a way to write $y=x_1+x_2+...+x_s$ where $x_1,x_2,...,x_s\in (\Z/5)^{\times(2)}$ so that
    \begin{equation*}
        \frac{1}{s}(f(x_1)+f(x_2)+...+f(x_s))\geq \E[f]>D_s\geq d_s.
    \end{equation*}\\

    \noindent\underline{When $s=7$:}

    By enumerating all possibilities, we observe that for $0,2,3\in \Z/5$, there are two ways to write them as a sum of $1$'s and $4$'s, so that the number of $1$'s is not less than the number of $4$'s in one way, and the number of $4$'s is not less than the number of $1$'s in the other way. Hence, the elements $0,2,3$ can be handled by the same reasoning in previous case $(s\geq 8)$.

    For $1,4\in \Z/5$, we have $1=1+1+1+1+4+4+4\;\mathrm{mod}\;5$ and $4=1+1+1+4+4+4+4\;\mathrm{mod}\;5$. The analysis for elements $1$ and $4$ is similar and we will demonstrate it for the element $1$. First, note that $f(1),f(4)>2D_7-1$ since $\E[f]=(f(1)+f(4))/2>D_7$. To proceed, we compute
    \begin{equation*}
            \frac{1}{7}(4f(1)+3f(4))>\frac{1}{7}(6D_7+(2D_7-1))=\frac{5}{7}=d_7.
    \end{equation*}\\

    \noindent\underline{When $s=6$:}

    By enumerating all possibilities, we observe that for $0,1,4\in \Z/5$, there are two ways to write them as a sum of $1$'s and $4$'s, so that the number of $1$'s is not less than the number of $4$'s in one way, and the number of $4$'s is not less than the number of $1$'s in the other way. Hence, the elements $0,1,4$ can be handled by the same reasoning in previous cases.

    For $2,3\in \Z/5$, we have $2=1+1+1+1+4+4\;\mathrm{mod}\;5$ and $3=1+1+4+4+4+4\;\mathrm{mod}\;5$. The analysis for elements $2$ and $3$ is similar and we will demonstrate it for the element $2$. First, note that $f(1),f(4)>2D_6-1$ since $\E[f]=(f(1)+f(4))/2>D_6$. To proceed, we compute
    \begin{equation*}
            \frac{1}{6}(4f(1)+2f(4))>\frac{1}{7}(4D_6+2(2D_6-1))=\frac{5}{6}=d_6.
    \end{equation*}\\

    \noindent\underline{When $s=5$:}

    Since the size of $(\Z/35)^{\times(2)}$ is $6$ and $\E[f]>D_5$, we have $f(x)>6D_5-5$ for any $x\in (\Z/35)^{\times(2)}$. Note that the $5$-fold sumset of $(\Z/35)^{\times(2)}$ is $\Z/35$, hence, for all $y\in \Z/35$, we can write $y=x_1+x_2+x_3+x_4+x_5$ for some $x_1,x_2,x_3,x_4,x_5\in (\Z/35)^{\times(2)}$. To complete the proof, we compute
    \begin{equation*}
        \frac{1}{5}(f(x_1)+f(x_2)+f(x_3)+f(x_4)+f(x_5))>\frac{1}{5}(5(6D_5-5))=\frac{9}{10}=d_5.
    \end{equation*}
\end{proof}

The following theorem combines Theorem \ref{main for almost all moduli} and Proposition \ref{main for small moduli}. The method for Theorem \ref{main for all moduli} is similar to that of Shao \cite[Proposition 3.1]{Shao} and Li and Pan \cite[Theorem 1.2]{LiPan}. It might be worth to point out that the proof of Theorem \ref{main for all moduli} is similar to the inductive step of the proof of Theorem \ref{main for almost all moduli}.

\begin{thm}\label{main for all moduli}
    Let $s\geq 5$ be an integer, and $W$ be a square free integer with $(24,W)=1$. Let $f:(\Z/24W)^{\times(2)}\longrightarrow[0,1]$ be a function with $\E[f]>D_s$. For all $y\in \Z/24W$ with $y\equiv s\; \mathrm{mod}\;24$, we have $y=x_1+x_2+...+x_s$ for some $x_1,x_2,...,x_s\in (\Z/24W)^{\times(2)}$ such that
    \begin{equation*}
        \text{$\frac{1}{s}(f(x_1)+f(x_2)+...+f(x_s))>d_s$ and $f(x_j)\neq 0$ for all $1\leq j \leq s$.}
    \end{equation*}
\end{thm}

\begin{proof}
    Without loss of generality, we may assume $5|W$ when $s\geq 6$ and $35|W$ when $s=5$. This is because if $W_1|W_2$ and Theorem \ref{main for all moduli} holds for modulus $W_2$, then Theorem \ref{main for all moduli} is also true for modulus $W_1$.

    Write $W=mW'$ where $m=5$ if $s\geq 6$, and $m=35$ if $s=5$. By Chinese remainder theorem, we can identify $(\Z/24W)^{\times(2)}$ with $(\Z/24)^{\times(2)}\times(\Z/m)^{\times(2)}\times(\Z/W')^{\times(2)}$. The elements in $(\Z/24W)^{\times(2)}$ can be conveniently written as $(1,x,x')$ where $x\in (\Z/m)^{\times(2)},x'\in (\Z/W')^{\times(2)}$, since $1$ is the only element in $(\Z/24)^{\times(2)}$.

    Define $g:(\Z/m)^{\times(2)}\longrightarrow[0,1]$ by setting
    \begin{equation*}
        g(x)=\E_{x'\in(\Z/W')^{\times(2)}}f(1,x,x')
    \end{equation*}
    for all $x\in(\Z/m)^{\times(2)}$. Note that we have $\E[g]=\E[f]>D_s$. By applying Proposition \ref{main for small moduli}, for all $(s,y,y')\in \Z/24W$, we have $y=x_1+x_2+...+x_s$ for some $x_1,x_2,...,x_s\in (\Z/m)^{\times(2)}$ such that
    \begin{equation*}
        \text{$\frac{1}{s}(g(x_1)+g(x_2)+...+g(x_s))>d_s$ and $g(x_j)\neq 0$ for all $1\leq j \leq s$.}
    \end{equation*}

    For each $1\leq j\leq s$, define $h_j:(\Z/W')^{\times(2)}\longrightarrow[0,1]$ by setting
    \begin{equation*}
        h_j(x')=f(1,x_j,x')
    \end{equation*}
    for all $x'\in (\Z/W')^{\times(2)}$. Note that we have $\E[h_j]=g(x_j)\neq 0$ for all $1\leq j\leq s$ and
    \begin{equation*}
        \frac{1}{s}(\E[h_1]+\E[h_2]+...+\E[h_s])=\frac{1}{s}(g(x_1)+g(x_2)+...+g(x_s))>d_s.
    \end{equation*}
    By applying Theorem \ref{main for almost all moduli}, we have $y'=x_1'+x_2'+...+x_s'$ for some $x_1',x_2',...,x_s'\in (\Z/W')^{\times(2)}$ such that
    \begin{equation*}
        \text{$\frac{1}{s}(h_1(x_1')+h_2(x_2')+...+h_s(x_s'))>d_s$ and $h_j(x_j')\neq 0$ for all $1\leq j \leq s$.}
    \end{equation*}
    which is the same as
    \begin{equation*}
        \text{$\frac{1}{s}(f(1,x_1,x_1')+f(1,x_2,x_2')+...+f(1,x_s,x_s'))>d_s$ and $f(1,x_j,x_j')\neq 0$ for all $1\leq j \leq s$.}
    \end{equation*}
    Hence, $(s,y,y')=(1,x_1,x_1')+(1,x_2,x_2')+...+(1,x_s,x_s')$ is the desired decomposition.
\end{proof}

Now, we have all ingredients needed to deduce Theorem \ref{main for primes} from the transference principle. Roughly speaking, Theorem \ref{main for all moduli} is essential for completing the $W$-trick. We state the general transference mechanism as follow.

\noindent\textbf{Hypothesis $(s,\mu)$:} \textit{Let $W$ be a square free integer with $(24,W)=1$, and $f:(\Z/24W)^{\times(2)}\longrightarrow[0,1]$ be a function with $\E[f]>\mu$. For all $y\in \Z/24W$ with $y\equiv s\; \mathrm{mod}\;24$, we have $y=x_1+x_2+...+x_s$ for some $x_1,x_2,...,x_s\in (\Z/24W)^{\times(2)}$ such that}
\begin{equation*}
    \textit{$\frac{1}{s}(f(x_1)+f(x_2)+...+f(x_s))\geq \frac{1}{2}$ and $f(x_j)\neq 0$ for all $1\leq j \leq s$.}
\end{equation*}

\begin{prop}\label{reduction to local problems}
    Let $s\geq 5$ be an integer, $0<\mu<1$ be a real number, and $A$ be a subset of primes with relative lower density $\delta_{\mathcal{P}}(A)>\sqrt{\mu}$. Suppose the Hypothesis $(s,\mu)$ is true, then, for all sufficiently large integers $y$ that are congruent to $s\;\mathrm{mod}\;24$, there exist primes $p_1,p_2,...,p_s\in A$ such that $y=p_1^2+p_2^2+...+p_s^2$.
\end{prop}

Theorem \ref{main for primes} indeed follows from Proposition \ref{reduction to local problems} since Theorem \ref{main for all moduli} implies that Hypothesis $(s,D_s)$ is true for $s\geq 5$. Proposition \ref{reduction to local problems} is quite close to Section 3 in \cite{LiPan}, where Li and Pan applied the transference principle to deduce their density version of Vinogradov three primes theorem from local results. Hence, the proof of Proposition \ref{reduction to local problems} only requires slight modifications to the proof in Section 3 in \cite{LiPan}, and it will be deferred to Appendix \ref{appendix 1}. 

\appendix
\section{Proof of Proposition \ref{Gauss}}\label{appendix}
For any odd prime $p$, let
\begin{equation*}
    \epsilon_p=\begin{cases}
        1&\text{if $p\equiv1$ mod $4$}\\
        i&\text{if $p\equiv3$ mod $4$}
    \end{cases}
\end{equation*}
First, we recall the following result on Gauss sums with reduced residue classes.

\begin{lem}\label{reduced Gauss sums}
    Let $p\geq3,n\geq1$ and $t\in \Z$ with $(t,p)=1$, then
    \begin{equation*}
        \sum_{y\in(\Z/p^n)^{\times}}e_{p^n}(ty^2)=\begin{dcases}
            \left(\frac{t}{p}\right)\epsilon_p\sqrt{p}-1&\text{when $n=1$}\\
            0&\text{when $n\geq2$}
        \end{dcases}
    \end{equation*}
    where $\left(\frac{\cdot}{p}\right)$ is the Legendre symbol.
\end{lem}

\begin{proof}
    The case that $n=1$ is more well-known, see \cite[Chapter 3, Section 3.5, Theorem 3.3]{ANT}. The cases that $n\geq 2$ can be derived from \cite[Chapter 3, Section 3.4, Lemma 3.1]{ANT}. Alternatively, one can calculate directly
    \begin{equation*}
        \begin{split}
            \sum_{y\in(\Z/p^n)^{\times}}e_{p^n}(ty^2)&=\sum_{\substack{z\in(\Z/p^{n-1})^{\times}\\w\in\Z/p}}e_{p^n}(t(z+p^{n-1}w)^2)\\
            &=\sum_{z\in(\Z/p^{n-1})^{\times}}e_{p^n}(tz^2)\sum_{w\in\Z/p}e_{p^n}(2tp^{n-1}zw)\\
            &=0
        \end{split}
    \end{equation*}
    since $z,t$ are coprime to $p$.
\end{proof}

\begin{proof}[Proof of Proposition \ref{Gauss}]
    We will partition the squares in $\Z/p^n$ according to the divisibility of $p$ powers. Let $0\leq e< n$, if $x\in(\Z/p^n)^{(2)}$ and $p^e||x$, then $e$ must be even and we can write $x=p^ex'$ for some $x'\in(\Z/p^{n-e})^{\times(2)}$.

    Write $t=p^vt'$ with $(t',p)=1$. We may assume $v<n$ otherwise the results are trivial. When $e<n-v$, we have
    \begin{equation*}
    \begin{split}
        \sum_{x'\in(\Z/p^{n-e})^{\times(2)}}e_{p^n}(tp^ex')&=\sum_{x'\in(\Z/p^{n-e})^{\times(2)}}e_{p^{n-v-e}}(t'x')\\
        &=p^v\sum_{x'\in(\Z/p^{n-v-e})^{\times(2)}}e_{p^{n-v-e}}(t'x')\\
        &=p^v\cdot\frac{1}{2}\sum_{y\in(\Z/p^{n-v-e})^{\times}}e_{p^{n-v-e}}(t'y^2)\\
        &=\begin{dcases}
            O(p^{v+1/2})&\text{when $n-v-e=1$}\\
            0&\text{when $n-v-e\geq2$}
        \end{dcases}
    \end{split}        
    \end{equation*}
    by Lemma \ref{reduced Gauss sums}. When $e\geq n-v$, we have
    \begin{equation*}
        \abs[\bigg]{\sum_{x'\in(\Z/p^{n-e})^{\times(2)}}e_{p^n}(tp^ex')}\leq p^{n-e}.
    \end{equation*}
    Ranging through all $e$ such that $e\geq n-v$, the contribution is at most
    \begin{equation*}
        p^v+p^{v-1}+...+p+1=O(p^v)
    \end{equation*}
    Combining the contributions from $e<n-v$ and $e\geq n-v$ completes the proof.
\end{proof}

\section{Proof of Proposition \ref{reduction to local problems}}\label{appendix 1}

In this section, we will closely follow Section 3 in \cite{LiPan} to prove Proposition \ref{reduction to local problems}.

First, we will define majorants that parametrize prime squares. These majorants have appeared in \cite{Chow}, \cite{Gao}, \cite{zhao2025densitytheoremprimesquares} and \cite{Tan}. Let $N,W\geq 2$ be integers with $W\leq \log N$, and $b\in (\Z/W)^{\times(2)}$. Denote the number of square roots of $b$ by $\sigma_W(b)$, that is
\begin{equation*}
    \sigma_W(b)=|\{x\in (\Z/W)^{\times}:x^2=b\}|
\end{equation*}
In fact, the value of $\sigma_W(b)$ is independent of $b$ and equal to $\varphi(W)/|(\Z/W)^{\times(2)}|$, so we will often abbreviate it as $\sigma_W$. Define a function $\nu_b:[N]\longrightarrow\R_{\geq 0}$ by setting
\begin{equation*}
    \nu_b(n)=\begin{dcases}
        \frac{\varphi(W)}{W\sigma_W(b)}(2p\log p)&\text{if $Wn+b=p^2$ for some prime $p$}\\
        0&\text{otherwise}
    \end{dcases}
\end{equation*}
for all $n\in [N]$. Roughly speaking, the majorant $\nu_b$ parametrizes weighted prime squares that are $b\;\mathrm{mod}\;W$ in $[WN]$.

Before proceeding further, we record the following well-known asymptotic formulae which would be useful later:
\begin{equation}\label{eqs: prime 1}
    \sum_{x\leq X}x\Lambda(x)=\frac{X^2}{2}+O(X^2\exp(-c\sqrt{\log X}))
\end{equation}
where $\Lambda$ is von Mangoldt function and $c$ is a positive absolute constant. Also, let $Q>0,X\geq 1$, and $a,r$ be positive integers with $r\leq (\log X)^Q$ and $(a,r)=1$, then
\begin{equation}\label{eqs: prime 2}
    \sum_{\substack{x\leq X\\x\equiv a\;\mathrm{mod}\;r}}x\Lambda(x)=\frac{X^2}{2\varphi(r)}+O(X^2\exp(-c_Q\sqrt{\log X}))
\end{equation}
where $c_Q$ is a positive absolute constant depending on $Q$. These asymptotic formulae can be proved by Siegel-Walfisz theorem and summation by parts.

By applying formula \eqref{eqs: prime 2}, one can show that
\begin{equation*}
    \sum_{n\in [N]}\nu_b(n)=N+O(N\exp(-c\sqrt{\log N}))
\end{equation*}
where $c>0$ is a constant. (The constant $c$ here might be different from the constant $c$ in formula \eqref{eqs: prime 1}, however, we will abuse the notation here and throughout.)

Now, let $A$ be a subset of primes with relative lower density $\delta_{\mathcal{P}}(A)>\sqrt{\mu}$. Let $N,W\geq 2$ be integers with $W\leq \log N$, and $b\in (\Z/W)^{\times(2)}$, we will define a function $f_{A,b}:[N]\longrightarrow\R_{\geq 0}$ by setting
\begin{equation*}
    f_{A,b}(n)=\begin{dcases}
        \frac{\varphi(W)}{W\sigma_W}(2p\log p)1_{A}(p)&\text{if $Wn+b=p^2$ for some prime $p$}\\
        0&\text{otherwise}
    \end{dcases}
\end{equation*}
for all $n\in [N]$. Roughly speaking, the function $f_{A,b}$ can be thought as a weighted, $W$-tricked version of characteristic function of $A$.

Let $y$ be a sufficiently large (depending on $A$ and $s$) integer such that $y\equiv s\;\mathrm{mod}\;24$. Let
\begin{equation*}
    w=\log\log\log y,\;\;W=8\prod_{3\leq p<w}p,\;\;N=\left\lfloor\frac{2y}{sW}\right\rfloor .
\end{equation*}
Note that we have $w=O(\log\log\log N)$ and hence $W\leq \log N$. We will work with $\nu_b,f_{A,b}$ defined as before based on this choice of $N,W$.

Let $\kappa>0$ be a sufficiently small parameter depending on $A$, for the sake of the definiteness, say $\kappa=(\delta_{\mathcal{P}}(A)-\sqrt{\mu})^2/10^{10}$. Define a function $f:(\Z/W)^{\times(2)}\longrightarrow[0,1]$ by setting
\begin{equation*}
    f(b)=\max\left\{0,\;\E_{n\in [N]}f_{A,b}(n)-\kappa\right\}
\end{equation*}
for all $b\in (\Z/W)^{\times(2)}$. Note that
\begin{equation*}
    \begin{split}
        &\sum_{b\in (\Z/W)^{\times(2)}}\sum_{n\in [N]}f_{A,b}(n)\\
        &=\frac{2\varphi(W)}{W\sigma_W}\sum_{\substack{p\in A\\W\leq p^2\leq W(N+1)\\p^2\in (\Z/W)^{\times(2)}}}p\log p\\
        &\geq \frac{2\varphi(W)}{W\sigma_W}\Bigg(\sum_{\substack{p\in A\\p\leq \sqrt{W(N+1)}}}p\log p-O(W)\Bigg)\\
        &\geq \frac{2\varphi(W)}{W\sigma_W}\Bigg(\sum_{p\leq (\sqrt{\mu}+2\sqrt{\kappa})\sqrt{W(N+1)}}p\log p-O(W)\Bigg)\;\;\;\quad\quad\quad\quad\quad\text{(since $\delta_{\mathcal{P}}(A)>\sqrt{\mu}$)}\\
        &\geq \frac{2\varphi(W)}{W\sigma_W}\left(\frac{(\sqrt{\mu}+2\sqrt{\kappa})^2WN}{2}-O(WN\exp(-c\sqrt{\log N}))\right)\;\;\quad\text{(by formula \eqref{eqs: prime 1})}\\
        &\geq \frac{\varphi(W)}{\sigma_W}\left((\sqrt{\mu}+2\sqrt{\kappa})^2N-O(N\exp(-c\sqrt{\log N}))\right)
    \end{split}
\end{equation*}
Thus, we have
\begin{equation*}
    \E_{b\in (\Z/W)^{\times(2)}}f(b)\geq \mu+\kappa
\end{equation*}
Even though the relative lower density of $A$ is greater than $\sqrt{\mu}$, the $\nu$-weighted density of $A$ is only greater than $\mu$, and this is the reason why the threshold in Theorem \ref{main for primes} is $\sqrt{D_s}$ instead of $D_s$.

By Hypothesis $(s,\mu)$, we have $y=b_1+b_2+...+b_s\;\mathrm{mod}\;W$ for some $b_1,b_2,...,b_s\in (\Z/W)^{\times(2)}$ such that
\begin{equation*}
    \textit{$\frac{1}{s}(f(b_1)+f(b_2)+...+f(b_s))\geq \frac{1}{2}$ and $f(b_j)\neq 0$ for all $1\leq j \leq s$.}
\end{equation*}
In particular, this implies that
\begin{equation*}
    \E[f_{A,b_1}]+\E[f_{A,b_2}]+...+\E[f_{A,b_s}]>\frac{s}{2}(1+\kappa)\;\text{and}\;\E[f_{A,b_j}]>\frac{\kappa}{2}\;\text{for all}\;1\leq j\leq s.
\end{equation*}

Let $y'=(y-b_1-b_2-...-b_s)/W$. The goal is to find $n_1,n_2,...,n_s\in [N]$ such that
\begin{equation*}
    \text{$y'=n_1+n_2+...+n_s$ and $f_{A,b_j}(n_j)>0$ for all $1\leq j\leq s$}
\end{equation*}
since $Wn_j+b_j$ will be a square of prime in $A$ when $f_{A,b_j}(n_j)>0$ for all $1\leq j\leq s$.

Now that we have completed the $W$-trick, for the rest of the details of the transference principle, we are able to find a well-formulated proposition in literature to cover them. The following proposition due to Salmensuu \cite[Proposition 3.9]{Salmensuu} is what we need.

\begin{prop}\label{transference lemma}
    Let $N\geq 1, s\geq 2$ be integers and $\epsilon\in (0,1)$. Let $f_j:[N]\longrightarrow\R_{\geq 0}, \nu_j:[N]\longrightarrow\R_{\geq 0}$ be functions with $f_j\leq \nu_j$ for $1\leq j\leq s$. Assume the following are true.
    \begin{enumerate}
        \item (Pseudorandomness) $\norm{\widehat{\nu}_j-\widehat{1}_{[N]}}_{\infty}\leq \eta N$ for some $\eta>0$ and all $1\leq j\leq s$
        \item (Restriction at $q$) $\norm{\widehat{f}_j}_q\leq KN^{1-1/q}$ for some $K>0,s-1<q<s$ and all $1\leq j\leq s$
        \item $\E[f_1]+\E[f_2]+...+\E[f_s]>\frac{s}{2}(1+\epsilon)$
        \item $\E[f_j]>\frac{\epsilon}{2}$ for all $1\leq j\leq s$
    \end{enumerate}
    Suppose $\eta$ is sufficiently small depending on $s,\epsilon,K,q$. Then, for any integers $y$ with
    \begin{equation*}
        y\in \left((1-(\epsilon/32)^2)\frac{sN}{2}, (1+(\epsilon/32))\frac{sN}{2}\right)
    \end{equation*}
    we have
    \begin{equation*}
        f_1*f_2*...*f_s(y)\geq c(\epsilon,s)N^{s-1}
    \end{equation*}
    where $c(\epsilon,s)>0$ is a constant that only depends on $\epsilon,s$.
\end{prop}

The proof of Proposition \ref{transference lemma} essentially involves embedding $[N]$ in $\Z/N'$ for some prime $N'\in \left[(1+(\kappa^2/100))\frac{y}{W}, (1+(2\kappa^2/100))\frac{y}{W}\right]$. Passing from $N$ to $N'$ loses density by a factor of $2/s$, and this is why the lower bound $1/2$ appears in Hypothesis $(s,\mu)$.

To apply Proposition \ref{transference lemma} to $f_{A,b_j},\nu_{b_j}$ for $1\leq j\leq s$ and complete the proof, it remains to verify the pseudorandomness and restriction at $q$. These types of estimates have been proved by Chow \cite{Chow}, Gao \cite{Gao} and Zhao \cite{zhao2025densitytheoremprimesquares}. The version we need is in \cite[Proposition 3.2, 3.3]{zhao2025densitytheoremprimesquares}, which states that
\begin{equation*}
    \norm{\widehat{\nu}_b-\widehat{1}_{[N]}}_{\infty}\ll_{\epsilon} \frac{N}{w^{1/2-\epsilon}}
\end{equation*}
for all $\epsilon>0$, and
\begin{equation*}
    \norm{\widehat{f}_{A,b}}_q\ll_q N^{1-1/q}
\end{equation*}
for all $q>4$.

\printbibliography %Prints bibliography

\end{document}